\documentclass[a4paper]{scrartcl}
\usepackage{a4wide}

\usepackage{amsmath,amssymb}
\usepackage{mathabx,bm}
\usepackage{enumerate}
\usepackage{xcolor}

\usepackage{authblk}

\usepackage{graphicx}

\usepackage{amsthm}
\usepackage{abstract}

\newtheorem{theorem}{Theorem}
\newtheorem{lemma}{Lemma}
\newtheorem{corollary}{Corollary}
\newtheorem{definition}{Definition}
\newtheorem{remark}{Remark}

\newtheorem{proposition}{Example}

\newcommand{\fcan}{f_\text{can}}

\newcommand{\fQ}{f_\mathcal{Q}}

\newcommand{\pgen}[1]{G_{#1}}
\newcommand{\gen}{L}
\newcommand{\genl}{L_{\text{Lan}}}

\newcommand{\R}{\mathbb{R}}

\newcommand{\eps}{\varepsilon}

\newcommand{\tr}{{\rm tr}\,}

\newcommand{\wrt}{with respect to }

\newcommand{\op}[1]{P^{#1}}
\newcommand{\opl}[1]{P_\text{Lan}^{#1}}

\newcommand{\Qspace}{\mathcal{Q}}
\newcommand{\Pspace}{\mathcal{P}}
\newcommand{\Xspace}{\mathcal{X}}
\newcommand{\Zspace}{\mathcal{Z}}
\newcommand{\Mspace}{\mathcal{M}}

\newcommand{\ep}{\varepsilon}
\newcommand{\proj}{\Pi}
\newcommand{\oproj}{\Pi^{\perp}}

\renewcommand{\bm}{}

\begin{document}
\title{Pseudo generators for under-resolved molecular dynamics}
\subtitle{What does the Smoluchowski equation tell us about the spatial dynamics of molecular systems?}

\renewcommand\Authfont{\fontsize{13}{14.4}\selectfont}
\renewcommand\Affilfont{\fontsize{11}{10.8}\itshape}

\author[1]{Andreas Bittracher\thanks{{bittrach@ma.tum.de}}}
\author[2]{Carsten Hartmann\thanks{{chartman@mi.fu-berlin.de}}}
\author[1]{Oliver Junge\thanks{{oj@tum.de}}}
\author[2]{P\'eter Koltai\thanks{{peter.koltai@fu-berlin.de}}}
\affil[1]{Fakult\"at f\"ur Mathematik, Technische Universit\"at M\"unchen, D-87548 Garching}
\affil[2]{Institut f\"ur Mathematik, Freie Universit\"at Berlin, D-14195 Berlin}

\date{\large\today}

\maketitle

\begin{abstract}
Many features of a molecule which are of physical interest (e.g. molecular conformations, reaction rates) are described in terms of its dynamics in configuration space. This article deals with the projection of molecular dynamics in phase space onto configuration space. Specifically, we study the situation that the phase space dynamics is governed by a stochastic Langevin equation and study its relation with the configurational Smoluchowski equation in the three different scaling regimes: Firstly, the Smoluchowski equations in non-Cartesian geometries are derived from the overdamped limit of the Langevin equation. Secondly, transfer operator methods are used to describe the metastable behaviour of the system at hand, and an explicit small-time asymptotics is derived on which the Smoluchowski equation turns out to govern the dynamics of the position coordinate (without any assumptions on the damping). By using an adequate reduction technique, these considerations are then extended to one-dimensional reaction coordinates. Thirdly, we sketch three different approaches to approximate the metastable dynamics based on time-local information only. 
\end{abstract} 
\section{Introduction}
\label{intro}

Langevin and Hamiltonian dynamics constitute established models for the analysis of biomolecular processes by classical molecular dynamics. While they describe the system at hand through the evolution of \emph{configuration} and \emph{momentum} coordinates, many properties of interest, such as metastable conformations, conformational transition rates or folding pathways, are merely characterized by the configurational dynamics or the dynamics of few collective variables, called \emph{reaction coordinates}, that span a low-dimensional submanifold of the configuration space (see, e.g., \cite{BeHu10,Din00,EVE04}). 

Both from a computational and modeling point of view it is very appealing to describe a molecular system just by its position (or reaction) coordinates, since this drastically reduces the dimensionality of the problem. Over decades, it has been of major interest to derive equations which govern the evolution of these coordinates either exactly~\cite{Zwa73,LuVE14}, or approximately with the smallest possible error~\cite{ChHaKu00,LeLe10}. One popular model for molecular dynamics in position space that comes under various names like \emph{overdamped Langevin dynamics},  \emph{Brownian dynamics}, \emph{Kramers equation} or Smoluchwski equation is obtained by the so-called Smoluchowski-Kramers approximation of the Langevin equation  \cite{Smo16,Kra40}; see also \cite{50yrs,Gar09} and the references therein. Yet it is unclear whether there are conditions beyond the asymptotic regime of the Kramers-Smoluchowski approximation (i.e.\ the high-friction limit), under which the Smoluchowski equation accurately captures e.g., the folding dynamics of a protein in terms of a one-dimensional reaction coordinate. For a general account of this topic we refer to \cite{100yrs}. 

 {  \paragraph{Aims and scope of this article.}
In this article, we discuss the accuracy of the Smoluchowski equation for the spatial dynamics of a molecular system under various parameter regimes where, in each case, our analysis departs from the \emph{Langevin equation} in phase space.\footnote{We use the terms \emph{spatial}, \emph{position(al)}, and \emph{configuration(al)} interchangeably, when referring to coordinates or dynamics.}  Our presentation of the topic is not claimed to be exhaustive; it rather reflects the authors' interests, and their wish to understand how the hierarchies of models used in molecular dynamics relate to each other. Parts of this article are based on the recent work \cite{BiKoJu14} by some of the authors, however, their analysis in the context of long time scales and in non-Cartesian geometries is new.  The main contribution of this article is that it sketches solutions to answer the following questions:}

\begin{enumerate}[(a)]
\item What is the appropriate generalization of the Smoluchowski equation in generalized (non-Cartesian) coordinates, to be used, for example, in reduced-order models of protein folding or polymers, and how is it related to a phase space description of a molecular system?

\item Is there a closed equation for the spatial dynamics on small time intervals if the underlying phase space dynamics is governed by a Langevin equation?

\item How well (and in which sense) is a system's metastable behaviour approximated by the Smoluchowski dynamics when the phase space dynamics is generated by Langevin dynamics? What are the time scale regimes on which the approximation of the metastable dynamics by the Smoluchowski equation may be used? 
\end{enumerate}

The manuscript is organized as follows: Section~\ref{sec:TrajectoryEnsembleViews} introduces the basic model of molecular dynamics in terms of deterministic and stochastic differential equations and describes an operator-based framework for the evolution of probability densities under these dynamics. This section also introduces the formulations of the stochastic equations in generalized coordinates in a non-Euclidean space. 
Section~\ref{sec:SpatialDynamicsMetastability} reviews the concept of metastability based on density fluctuations in position space and establishes a connection between Langevin and Smoluchowski dynamics on short time scales. A numerically exploitable scheme which replaces the complicated position space density transport by a rescaled Smoluchowski transport is described, along with asymptotic error estimates.
Section~\ref{sec:largetime} reviews the approximation quality of these methods, gives improved error estimates and discusses the extension to longer times scales. A summary and possible future directions are given in Section~\ref{discussion}.

\section{Trajectory- and ensemble-based views} \label{sec:TrajectoryEnsembleViews}


We consider a dynamical system described by $d=3n$ positional degrees of freedom that represent a system of $n$ particles. Let $\Qspace\subset\mathbb{R}^d$ denote the corresponding configuration space and $V(q)$ the potential energy of a given particle configuration $q\in\Qspace$ where we assume that the function $V\colon\Qspace\rightarrow \mathbb{R}$ is at least twice continuously differentiable, polynomially growing at infinity and bounded from below. 

\subsection{Models for molecular dynamics}
\label{ssec:models}

We introduce three typical models for molecular dynamics. The simplest model to describe the motion $(q_{t})_{t\ge 0}$ of the particles in vacuum, i.e.~without external influences like a solvent is given by \emph{Hamilton's equations}
\begin{equation}\label{hamiltondynamics}
\begin{aligned}
\frac{d\bm{q}}{dt}&=\nabla_{\bm p} H(\bm{q},\bm{p})\\
\frac{d\bm{p}}{dt}&=-\nabla_{\bm q} H(\bm{q},\bm{p})\,,
\end{aligned}
\end{equation}
where $p\in\Pspace=\mathbb{R}^d$ denotes the vector of conjugate particle momenta, and 
\begin{equation*}
H = \frac{1}{2}p\cdot M^{-1}p + V(q)
\end{equation*}
is the Hamiltonian (total energy) of the system, with $M={\rm diag}(m_{1},\ldots,m_{d})$ denoting the mass matrix. Depending on the type of system or when transformed to generalized coordinates, the mass matrix $M$ can be a general symmetric positive-definite, possibly position-dependent matrix.

In the presence of a heat bath or solvent, one typically adds a drift-diffusion term, arriving at the \emph{Langevin equation}, 
\begin{equation}\label{langevindynamics}
\begin{aligned}
\frac{d\bm{q}}{dt} &= \nabla_{\bm p} H(\bm{q},\bm{p})\\
\frac{d\bm{p}}{dt} &= -\nabla_{\bm q} H(\bm{q},\bm{p}) - \gamma\nabla_{\bm p} H(\bm{q},\bm{p}) + \sigma \zeta_t~.
\end{aligned}
\end{equation}
The term $-\gamma\nabla_{\bm p} H = -\gamma M^{-1}\bm{p}$, with $\gamma\in\R^{d\times d}$ being symmetric positive definite, models the drag through the solvent, the term $\sigma \zeta_t$ accounts for random collisions with the solvent particles~\cite[Chap.~4]{Nel63}. Here, $(\zeta_t)_{t\ge 0}$ is an uncorrelated, zero-mean white noise process that can be formally interpreted as the (generalized) derivative of a standard $d$-dimensional Brownian motion, and $\sigma\sigma^{T}\in\R^{d\times d}$ is the noise covariance matrix. In order to keep the system at a constant average kinetic energy, damping and excitation have to be balanced, which is ensured by assuming that noise and drag coefficients satisfy the fluctuation-dissipation relation 
\[
2\gamma=\beta\sigma\sigma^{T}\,,
\]
where $\beta>0$ is the inverse temperature in the system. Choosing $\gamma$ or $\sigma$ is a modelling issue and thus depends on the particular problem at hand. As we will see later on, both $\gamma$ and $\sigma$ may even be position dependent. 
%

If the friction in the system is uniformly large, i.e. $v\cdot\gamma v\gg v\cdot M v$ for all $v\in\R^{d}$, the Langevin equation can be replaced by the  \emph{Smoluchowski equation}
\begin{equation}\label{smoluchowskidynamics}
\gamma \frac{d\bm{q}}{dt} = -\nabla V(\bm{q}) + \sigma \zeta_t\,,
\end{equation}
or, using the common notation for It\^o stochastic differential equations (see \cite{Oks98}),
\begin{equation*}
\gamma d\bm{q}_{t} = -\nabla V(\bm{q}_{t})dt + \sigma d\bm{w}_{t}\,,
\end{equation*}
where $w_{t}$ is a standard Brownian motion in $\R^{d}$.\footnote{{  In the following, we will interpret stochastic differential equations such as (\ref{langevindynamics}) or (\ref{smoluchowskidynamics}) in the sense of It\^o, which implies that, for stochastic differentials such as $dq_{t}$, a generalized  chain rule known as \emph{It\^o's formula} or \emph{It\^o's Lemma} applies \cite[Theorem~4.2.1]{Oks98}.}} {  The Smoluchowski equation is also termed \emph{overdamped Langevin equation} and can be derived 
 from (\ref{langevindynamics}) by letting $v\cdot\gamma v\to\infty$ under a suitable rescaling of time; for a precise statement, see \cite[Theorem~10.1]{Nel63} or the derivation given below on pp.~\pageref{eps-1}--\pageref{convgHFL}.}\\

In some cases a description of the stochastic dynamics in a different coordinate system is needed. The aim of this subsection therefore is to derive the Smoluchowski equation in generalized coordinates. This is most conveniently done by resorting to the canonical  
form of the Langevin equation (\ref{langevindynamics}) that we will state first.

\paragraph{Langevin equation in generalized coordinates.}

To state the Langevin equation (\ref{langevindynamics}) in canonical form, we consider a diffeomorphism $\Phi\colon \Qspace'\to \Qspace$ between configuration spaces that has a cotangent lift $T^{*}\Phi\colon\Xspace\to\Xspace'$ given by
\[
(\bm{q},\bm{p})\mapsto \left(\Phi^{-1}(\bm{q}),\,((\nabla \Phi\circ\Phi^{-1})(\bm{q}))^{T}\bm{p}\right)\,.
\]
Using It\^o's formula \cite[Theorem~4.2.1]{Oks98}, the Langevin equation (\ref{langevindynamics}) can be written in the new configuration variables $\bm{u}=\Phi^{-1}(\bm{q})$ and their conjugate momenta $\bm{v}=((\nabla \Phi\circ\Phi^{-1})(\bm{q}))^{T}\bm{p}$: introducing the new (possibly position-dependent) drag and noise coefficients by
\begin{equation}\label{langevincoefficientsCan}
\tilde{\gamma} = \nabla \Phi^{T}\gamma \nabla \Phi\,,\quad \tilde{\sigma}=\nabla \Phi^{T}\sigma\,,
\end{equation}
the Langevin equation can be recast as \cite{HartDiss,HaYa13}
\begin{equation}\label{langevindyanamicsCan}
\begin{aligned}
\frac{d\bm{u}}{dt}  & = \nabla_{\bm{v}}\tilde{H}(\bm{u},\bm{v})\\
\frac{d\bm{v}}{dt}  & = -\nabla_{\bm{u}}\tilde{H}(\bm{u},\bm{v}) - \tilde{\gamma}(\bm{u})\nabla_{\bm{v}}\tilde{H}(\bm{u},\bm{v}) + \tilde{\sigma}(\bm{u})\zeta_t\,.
\end{aligned}
\end{equation}
Here $\tilde{H}$ denotes the push-forward of the Hamiltonian $H$ to the new coordinate system,  
\[
\tilde{H}(\bm{u},\bm{v}) = \frac{1}{2}\bm{v}\cdot (G(\bm{u}))^{-1}\bm{v} + \tilde{V}(\bm{u})\,,
\]
with $\tilde{V}=V\circ\Phi^{-1}$ and $G=\nabla\Phi^{T} M\nabla\Phi$ being the mass metric tensor induced by the transformation $\Phi$. 

It can be readily seen that, when the original drag and noise coefficients satisfy the fluctuation-dissipation relation, then so do the transformed coefficients:  
\begin{equation}\label{fdrelation}
2\tilde{\gamma}=\beta\tilde{\sigma}\tilde{\sigma}^{T}\,.
\end{equation}

\paragraph{Derivation of the Smoluchowski equation in generalized coordinates.}

It is now possible to derive the Smoluchowski equation in generalized coordinates from the 
canonical Langevin dynamics \eqref{langevindyanamicsCan} using formal asymptotics. 
To this end, let us scale the original drag and 
noise coefficients according to $\gamma\mapsto\gamma/\eps$ and 
$\sigma\mapsto\sigma/\sqrt{\eps}$ where $\eps>0$ is a small parameter. 
Clearly, the scaling preserves the fluctuation-dissipation relation (\ref{fdrelation}), and it leads to the Langevin equation
\begin{equation}\label{langevindynamicsEps}
\begin{aligned}
\frac{d\bm{u}}{dt}  & = \nabla_{\bm{v}}H(\bm{u},\bm{v})\\
\frac{d\bm{v}}{dt}  & = -\nabla_{\bm{u}}H(\bm{u},\bm{v}) - \frac{1}{\eps}\gamma(\bm{u})\nabla_{\bm{v}} H(\bm{u},\bm{v}) + \frac{1}{\sqrt{\eps}}\sigma(\bm{u})\zeta_t\,.
\end{aligned}
\end{equation}
For notational convenience, we have dropped the twiddle signs on the transformed Hamiltonian $H$ and the coefficients $\gamma$ and $\sigma$. 
To study the $\eps\to 0$ limit of (\ref{langevindynamicsEps}) we seek a 
perturbative expansion of the associated 
backward Kolmogorov equation\footnote{The Kolmogorov backward equation is a partial differential equation governing the evolution of observables. Specifically, if $(X_t)_{t\ge 0}$ is the solution of an It\^o stochastic differential equation, such as (\ref{langevindynamicsEps}), then, for any integrable function $\phi_0\colon\mathcal{X}\to\R$,
\[
\phi(x,t)=\mathbb{E}_x[\phi_0(X_t)]
\]
satisfies $\partial_t\phi = A\phi$ with initial condition $\phi(x,0)=\phi_{0}(x)$. Here $A$ is the infinitesimal generator associated with the process $(X_{t})_{t\ge 0}$, and $\mathbb{E}_x[\cdot]$ denotes the conditional expectation over all realizations of the process starting at $X_0 = x$. We introduce the dual concept of  \emph{forward equation} and the corresponding generator in Section~\ref{ssec:transferoperator}; cf.~also \cite[Section~7.3]{Oks98}.
}
\begin{eqnarray}\label{bweEps}
\partial_t \phi^\eps(\bm{u},\bm{v},t) = A_{\text{Lan}} \phi^\eps(\bm{u},\bm{v},t)\,,\quad \phi^\eps(\bm{u},\bm{v},0)=\phi_{0}(\bm{u},\bm{v})\,
\end{eqnarray}
following the approach described 
in \cite{Pap76,Pav08}.  To begin with, we notice that 
the backward operator $A_{\text{Lan}}$ in (\ref{bweEps}) admits the decomposition  (see also p.~\pageref{langevingenerator} below)
\begin{eqnarray*}
A_{\text{Lan}} = A_{\text{Ham}} + \frac{1}{\eps}A_{\text{OU}}\,,
\end{eqnarray*}
with 
\[
A_{\text{Ham}} = \nabla_{\bm{v}} H\cdot\nabla_{u} - \nabla_{\bm{u}}H\cdot\nabla_{v}
\]
and 
\[
A_{\text{OU}} = \frac{1}{2}\sigma\sigma^{T}\colon\nabla_{\bm{v}}^{2} -  (\gamma G^{-1}\bm{v})\cdot\nabla_{\bm{v}}
\]
We consider a perturbative solution of (\ref{bweEps}) that is of the form
\begin{eqnarray*}
\phi^\eps\,=\, \phi_{0} + \eps \phi_{1} + \eps^2 \phi_{2} + \ldots
\end{eqnarray*}
with $\phi_{i}=\phi_{i}(\bm{u},\bm{v	},t)$. Inserting the ansatz into the 
backward equation and equating powers of $\eps$ we obtain a hierarchy 
of equations, the first three of which read
\begin{eqnarray}
\label{eps-1}
 A_{\text{OU}}\phi_{0} &=& 0 \\ \label{eps0}
 A_{\text{OU}}\phi_{1} &=& \partial_{t} \phi_{0} -  A_{\text{Ham}}\phi_{0}\\ \label{eps1}
 A_{\text{OU}}\phi_{2}&=& \partial_{t} \phi_{1} -  A_{\text{Ham}}\phi_{1}\,.
\end{eqnarray}
Note that $A_{\text{OU}}$ is a second-order differential operator in $\bm{v}$ with $\bm{u}$ 
appearing only as a parameter. By the assumption that $\gamma(\cdot)$ is symmetric positive definite with uniformly bounded inverse, (\ref{eps-1}) implies that $\phi_{0}$ does not depend on $\bm{v}$.  
By a standard closure argument (a.k.a.~\emph{centering condition}), it thus follows that $\partial_{t}\phi_{0}=0$. 

Closely inspecting the resulting equations (\ref{eps0})--(\ref{eps1}), the next nontrivial term, $\phi_{1}$, is found to be the 
solution of the backward equation
\begin{equation}\label{bweEff}
\partial_t \psi= -\int_{\R^{d}} \left(A_{\text{Ham}}A_{\text{OU}}^{-1}A_{\text{Ham}}\psi\right)\varrho_{\bm{u}}(\bm{v})d\bm{v}\,,
\end{equation}
where $\psi=\psi(\bm{u},t)$ is independent of $\bm{v}$, and $\varrho_{\bm u}$ is the solution to $A^{*}_{\text{OU}}\varrho_{\bm{u}}=0$, with $A^{*}_{\text{OU}}$ being the formal $\mathcal{L}^{2}$ adjoint of $A_{\text{OU}}$. Equation (\ref{bweEff}) must be equipped with a suitable initial condition $\psi(\bm{u},0)=\psi_{0}(\bm{u})$. 

Before we evaluate the right hand side of (\ref{bweEff}), a few remarks are in order: 

\begin{enumerate}

\item The function $\varrho_{\bm{u}}$ in (\ref{bweEff}) is the unique invariant probability density \wrt the Lebesgue measure on the momentum space $\Pspace' = T^{*}_{\bm{u}}\Qspace'$ (that we can identify with $\R^{d}$) of the Ornstein--Uhlenbeck process generated by $A_{\text{OU}}$ . It is given by 
\begin{eqnarray}\label{rhou}
\varrho_{\bm{u}}(\bm{v}) = \left(\frac{\beta}{2\pi}\right)^{n/2} \left(\det(G(\bm{u}))\right)^{-1/2}
\exp\left(-\frac{\beta}{2}\bm{v} \cdot (G(\bm{u}))^{-1}\bm{v}\right)\,
\end{eqnarray}
and satisfies $A^{*}_{\text{OU}}\varrho_{\bm{u}}=0$. \\

\item The inverse of the operator  $A_{\text{OU}}$ in (\ref{bweEff}) is only unambiguously defined when it is acting on functions that are in the range of $A_{\text{OU}}$. {  By the Fredholm alternative (see \cite[Sections 3.12 \& 4.3]{Zei95}), the range of a linear continuous operator on some Banach space is the orthogonal complement of the kernel of its adjoint: 
\[
{\rm ran}(A_{\text{OU}}) = \left({\rm ker}(A_{\text{OU}}^{*})\right)^{\perp} = V_{0}\,,
\]
where
\[
V_{0}:=\left\{f\in \mathcal{L}^{2}_{\mu}(\Xspace')\colon \mathcal \int_{\Pspace}f(\bm{v}, \bm{u})\varrho_{\bm{u}}(\bm{v}) d\bm{v}=0\right\}\subset  \mathcal{L}^{2}_{\mu}(\Xspace')\,.
\]  
That is to say that the linear equation $A_{\text{OU}}\Phi = c$ has a solution, if and only if $c$ averages to zero under $\varrho_{u}$.
Note that $c=A_{\text{Ham}}\psi$ is linear in the momenta $\bm{v}$, hence $A_{\text{Ham}}\psi\in V_{0}$. As a consequence, $A_{\text{OU}}^{-1}$ in (\ref{bweEff}) is well defined.} \\
\end{enumerate}
{  
The formal expansion (\ref{eps-1})--(\ref{bweEff}) now suggests that the solution of the Langevin-based backward equation (\ref{bweEps}) and the solution to the limiting system (\ref{bweEff}) satisfy  
\begin{equation}\label{convgHFL}
\|\phi^{\eps}(\cdot,t) - \psi(\cdot,t/\eps) \|_{V} \to 0\,,\quad \eps\to 0\,.
\end{equation}
for some suitable norm on $V\subset  \mathcal{L}^{2}_{\mu}$. Indeed, standard results from homogenization theory for parabolic partial differential equations (e.g. \cite{Pap76,PaVe05}) imply that, under certain regularity assumptions on the coefficients, the convergence is uniform in $\Xspace'\times(0,T)$ for any finite $T$ (cf.~also \cite[Theorem~10.1]{Nel63}). }

As we show in the appendix, the operator on the right hand side of (\ref{bweEff}) reads
\begin{equation}\label{smoluchowskiGen}
\bar{A} = \beta^{-1}\tilde{\Delta} - {\nabla} V\cdot \tilde{\nabla}\,,
\end{equation}
where
\begin{equation*}
\tilde{\nabla}  = \gamma^{-1}  \nabla \quad\mbox{ and }\quad
\tilde{\Delta}=\frac{1}{\sqrt{\det \gamma}} \nabla\cdot
\left(\sqrt{\det \gamma}\, \gamma^{-1} \nabla\,\right)\,,
\end{equation*}
denote gradient and Laplace-Beltrami operator \wrt $\gamma$. 
The differential operator $\bar{A}$ has a straighforward 
interpretation as the infinitesimal (backward) generator of the Smoluchowski dynamics on the configuration space 
$\Qspace'\subset\Xspace'$, with the position dependent drag matrix acting as metric 
tensor on $\Qspace'$. 
Alternatively, one may regard $\bar{A}$ as the generator of the Smoluchowski 
dynamics on a Riemannian manifold $\Qspace'$ endowed with the metric tensor 
$h=\nabla\Phi^{T}\nabla\Phi$ and a position dependent friction matrix $\gamma$.  
Our findings are summarized in the next lemma.

\begin{lemma}
The Smoluchowski equation in generalized coordinates reads 
\begin{equation}\label{smoluchowskidynamicsGen}
\gamma(\bm{u}) \frac{d\bm{u}}{dt} =  - \nabla V(\bm{u}) + g(\bm{u}) 
+ \sigma(\bm{u}) \zeta_t\,,\;\;\bm{u}_{0}=\bm{u}\,,
\end{equation}
where $g=(g_{1},\ldots,g_{n})$ has the entries 
\begin{eqnarray}\label{ito}
g_i = \frac{1}{\beta}\sum_{j,k}\gamma_{ij}\frac{1}{\sqrt{\det \gamma}}\frac{\partial}{\partial u_k}
\left(\sqrt{\det \gamma}\,\gamma^{kj}\right)\,.
\end{eqnarray}
\end{lemma}

\begin{remark}
The additional drift term $g$ in the Smoluchowski equation is due to the geometry of the configuration manifold $\Qspace'$ and the  
interpretation of the Smoluchowski equation in the sense of It\^o. Formally, it can be seen to be related with the first 
order derivative in the expression for the  Laplace Beltrami operator in (\ref{smoluchowskiGen}). 
\end{remark}

\begin{remark}
The Smoluchowski equation (\ref{smoluchowskidynamicsGen}) in generalized coordinates likewise follows by transforming the original Smoluchowski 
equation (\ref{smoluchowskidynamics}) in Cartesian coordinates to the new coordinate system using It\^o's Lemma \cite[Theorem~4.2.1]{Oks98}. As a consequence, the stochastic convergence of the spatial component of the high friction Langevin equation to the solution of the (time-rescaled) Smoluchowski equation that has been proved in~\cite[Theorem~10.1]{Nel63} should be inherited by its non-Cartesian counterpart.     
\end{remark}

\subsection{The transfer operator}\label{ssec:transferoperator}

We shall now examine how probability densities evolve under the Langevin or the Smoluchowski dynamics. To this end, call $\bm{x}=(\bm{q},\bm{p})$ or $\bm{x}=\bm{q}$ the state vector, depending on which type of dynamics is considered, and $\Xspace=\Qspace\times\Pspace$ or $\Xspace=\Qspace$ the phase space or state space, respectively. For any given $\bm{x}_0\in\Xspace$, we seek the probability density $f_t$ of $\bm x_t$ for some $t>0$ \wrt the natural (Liouville or Lebesgue) measure on $\Xspace$. Now let $B\subset\Xspace$ be any measurable subset of $\Xspace$, and define the stochastic transition function\footnote{It is common to denote the transition function and transition probabilities by~$p$. We hope that the clash in the notation with the conjugate momenta is not going to confuse the reader, since the transition function and -probabilities are always going to be functions of three variables.} of the dynamics $(x_{t})_{t\ge 0}$ by 
\begin{equation}\label{transitionkernel}
p(t,x,B) = \mathrm{Prob}\left[ \bm{x}_t\in B\,\vert\, \bm{x}_0=x\right],
\end{equation}
Further call  
\begin{equation}\label{transitionprobab}
p_{\mu}(t,A,B) := \mathrm{Prob}_{\mu}\left[\bm x_t\in B\,\vert\,\bm x_0\in A\right]
\end{equation}
the \emph{transition probability} between the measurable sets $A\subset\Xspace$ and $B\subset\Xspace$, where $\mathrm{Prob}_{\mu}$ indicates that the initial condition is distributed according to a probability measure~$\mu$. For the long term macroscopic behaviour of the system, sets $A$ play an important role for which~$p_{\mu}(t,A,A)\approx 1$ for some physically relevant measure~$\mu$ and some characteristic times~$t>0$.

The transition probability $p(t,x,B)$ can be derived from a stochastic transition kernel or \emph{transition density} $\Psi$ via 
\begin{equation}\label{stochastickernel}
	p(t,x,B) = \int_B\Psi(t,x,y)\,dy, 
\end{equation}
where $\Psi$ is the fundamental solution of the Fokker-Planck equation (\ref{fokkerplanckequationito}). Existence and uniqueness of the transition density $\Psi$ follow relatively under mild conditions (see e.g.~\cite[Chap.~11.7]{Las94}).

Now let an initial density $f_0=d\mu/d\bm{x}$ be given.  The density $f_t$ describing the distribution of the system at time $t>0$ is then given by 
\begin{equation}\label{stochtransport}
f_t(y) = \int_\Xspace f_0(x)\Psi(t,x,y)\,d\mu(x)
\end{equation}
Equation (\ref{stochtransport}) can be seen as the definition of the \emph{transfer operator with lag time~$t$}: 
\[
\op{t} f_{0}(x) := f_t(x).
\]
By linearity we can extend the definition of~$\op{t}$ from probability densities to arbitrary integrable functions, and it will be convenient in what follows to consider the transfer operator as a family of linear maps $\op{t}\colon \mathcal{L}^{1}_{\mu}(\Xspace)\to  \mathcal{L}^{1}_{\mu}(\Xspace)$. This  family of linear operators has the \emph{Chapman-Kolmogorov} (or \emph{semigroup}) property:
\begin{enumerate}[(i)]
\item $\op{0} f=f$,
\item $\op{t+s}f=\op{t}\big(\op{s}f\big)$ for all $s,t\ge0$.
\end{enumerate}
We also have non-expansiveness in the induced operator norm, $\|\op{t}\| \le 1$, and positivity, $\op{t}f\ge 0$ for $f\ge 0$.

The transition probabilities (\ref{transitionprobab}) can be conveniently expressed in terms of the transfer operator. If we define the scalar product on the space $\mathcal{L}^{2}_{\mu}(\Xspace)$ of square integrable functions by 
$$
\langle f,g\rangle_\mu=\int_\Xspace f(x) g(x)\,d\mu(x)\,,
$$
then, for any measurable set $A$ with $\mu(A)>0$,  
\[
p_{\mu}(t,A,B) = \frac{1}{\mu(A)}\int_B \op{t}\chi_{A}\,d\mu = \frac{1}{\mu(A)}\int_\Qspace \op{t}\chi_{A}\chi_{B}\,d\mu=\frac{\langle P^t\chi_A,\chi_B\rangle_{\mu}}{\langle \chi_A,\chi_A\rangle_{\mu}}
\]
with~$\chi$ being the indicator function.

\paragraph{The forward generator.}

The semigroup property means that~$\op{t}$ is ``memoryless'', i.e.\ that \eqref{hamiltondynamics}--\eqref{smoluchowskidynamics} generate a Markov process. Noting that  
\[
\op{t} = \left(\op{t/n}\right)^n\,,
\] 
we may conclude all relevant information about the density transport is already contained in $P^\tau$ for \emph{arbitrarily small} $\tau$.
This is formalized by looking at the operator 
\begin{equation}\label{generator}
\gen f=\lim_{\tau\rightarrow0}\frac{\op{\tau}f-f}{\tau}
\end{equation}
that is defined for all $f$, for which the limit exists. $\gen$ is called the \emph{forward generator} or \emph{infinitesimal generator} of $\op{t}$. 

From its general form for arbitrary It\^o diffusions \cite[p.~282]{Kar91}, we can derive the generator for our dynamics. For the Hamiltonian dynamics \eqref{hamiltondynamics} and functions $f\in C^{1}(\Xspace)$, where $C^{1}$ is equipped with the supremum norm, the operator $L$ is given by
\begin{equation}\label{liouvillian}
L_{\text{Ham}} =  \nabla_{q}H\cdot\nabla_{p} - \nabla_{p}H\cdot \nabla_{q}\,,
\end{equation}
where the dot denotes the Euclidean inner product, and $\nabla_q$, $\nabla_{p}$ are the gradients with respect to~$q$ or $p$. 
In case of Langevin dynamics (\ref{langevindynamics}) and $f\in C^{2}(\Xspace)$, we have 
\begin{equation}\label{langevingenerator}
L_{\text{Lan}} = L_{\text{Ham}} + \frac{1}{2}\sigma\sigma^{T}\colon\nabla_{p}^{2}  + \gamma\nabla_{p}H\cdot\nabla_{p} + \gamma\colon\nabla_{p}^{2}H\,,
\end{equation}
where the notation $A\colon B:=\tr(A^{T}B)$ denotes the inner product between matrices $A,B\in\R^{d\times d}$, and $\nabla_{p}^{2}$ denotes the matrix of second derivatives \wrt  $p$.    
Finally, the generator of the Smoluchowski dynamics (\ref{smoluchowskidynamics}) reads
\begin{equation}\label{smoluchowskigenerator}
L_{\text{Smol}} = \beta^{-1}\gamma^{-1}\colon\nabla_{q}^{2} + \left(\gamma^{-1}\nabla_{q}V\right)\cdot\nabla_{q} + \gamma^{-1}\colon\nabla_q^{2}V\,,
\end{equation}
with $\nabla_{q}^{2}$ being the matrix of second derivatives in $q$, and we have used that $2\gamma=\beta\sigma\sigma^{T}$.

\paragraph{Fokker-Planck equations and invariant measures.}

By definition of the forward generator, the evolution of probability densities $f_{t}$ associated with any of the stochastic dynamics \eqref{langevindynamics}--\eqref{smoluchowskidynamics} is described by a parabolic transport equation of the form  
\begin{equation}\label{fokkerplanckequationito}
\partial_tf_t = L  f_t\,,\quad f_{t=0}(x)=g(x)\,,
\end{equation}
that are called \emph{Kolmogorov forward equations} or \emph{Fokker--Plack equations}~\cite{Kar91}, with $L$ being either $L_{\text{Lan}}$ or $L_{\text{Smol}}$. When $\gamma=0$, then the Fokker-Planck equation with $L_{\text{Lan}}$ turns into the Liouville equation that describes the transport of probability densities under the Hamiltonian dynamics (\ref{hamiltondynamics}). 


Probability measures that are invariant under the dynamics play a prominent role. The corresponding densities are fixed points of $\op{t}$ for any $t\ge0$, and equation~\eqref{generator} implies that they lie in the kernel of~$\gen$. For the stochastic processes considered here, the invariant density can be shown to be unique (cf.~\cite{MaSt02}). For the Langevin dynamics (\ref{langevindynamics}), the unique invariant probability density is the  \emph{canonical density} 
\begin{equation}\label{canonicaldensity}
\begin{aligned}
\fcan(q,p) & = \frac{1}{Z}\exp\left(-\beta H(q,p)\right)\\
& =\underbrace{\frac{1}{Z_\Pspace}\exp\left(-\frac{\beta}{2}p\cdot M^{-1} p\right)}_{=:f_\Pspace(p)}~\underbrace{\frac{1}{Z_\Qspace}\exp\left(-\beta V(q)\right)}_{=:f_\Qspace(q)},
\end{aligned}
\end{equation}
with $Z=Z_{\Pspace}Z_{\Qspace}$ and 
\[
Z_\Pspace=\int_\Pspace\exp\left(-\frac{\beta}{2}p\cdot M^{-1} p\right)dp\,,\quad  Z_\Qspace=\int_\Pspace\exp\big(-\beta V(q)\big)dq\,.
\] 
For the Smoluchowski dynamics (\ref{smoluchowskidynamics}), the unique invariant measure has the density $f_{\Qspace}(q)$, which is called the \emph{Boltzmann density} or \emph{Gibbs-Boltzmann density}. (We assume throughout that $\exp(-\beta V)$ is integrable).

Under fairly mild assumptions on the potential $V$, the invariant densities can be shown to be the unique asymptotically stable fixed point of $P^{t}$, which implies that $P^{t}f_{0}$ converges to the stationary distribution for any initial density $f_0$ \cite{MaStHi02}. The Liouville equation associated with the Hamiltonian dynamics (\ref{hamiltondynamics}) is known to have infinitely many stationary solutions, among which is $\fcan$.   

\begin{remark}

It can be readily seen that the Smoluchowski dynamics (\ref{smoluchowskidynamicsGen}) in generalized coordinates $\bm{u}\in\Qspace'$
has the unique invariant probability measure 
\begin{eqnarray}\label{smoluchowskimeasure}
d\rho(\bm{u}) = (f_{\Qspace}\circ\Phi)(\bm{u}) \,d\Sigma(\bm{u})\,,
\end{eqnarray}
with 
\[
d\Sigma(\bm{u})=\sqrt{\det h(\bm{u})}\,d\bm{u}
\]
being the 
Riemannian volume element on $\Qspace'$ where $h(\bm{u})=(\nabla\Phi^{T}\nabla\Phi)(\bm{u})$ is the corresponding metric tensor. Note that (\ref {smoluchowskimeasure}) is simply the pullback of the Boltzmann distribution in Cartesian coordinates by the coordinate transformation $\Phi$. As a consequence, $d\rho/d\bm{u}$ is the $\Qspace'$-marginal of the canonical density $f_{\text{can}}$. To see this, replace the metric tensor $h$ on $\Qspace'$ by the generalized mass matrix  $G=\nabla\Phi^{T}M\nabla\Phi$ or the corresponding expression for the friction coefficient $\gamma$. This does not change the invariant measure as the constant mass or drag matrices cancel out.


\end{remark}

\subsection{More on semigroups and their generators}

Before we proceed, let us recall two results relating the transfer operator semigroup and its generator. For our purposes the main connection between them is given by the following
\begin{theorem}[Spectral mapping theorem~\cite{Paz83}]\label{spectralmappingtheorem} Let $\mathcal{X}$ be a Banach space, $T^t:\mathcal{X}\to\mathcal{X}$, $t\ge0$, a $C_0$ semigroup of bounded linear operators (i.e.\ $T^tf\to f$ as $t\to0$ for every $f\in\mathcal{X}$, and~$T^t$ bounded for every~$t$), and let $A$ be its infinitesimal generator. Then
$$
e^{t\sigma_{\bullet}(A)}\subset \sigma_{\bullet}(T^t)\subset e^{t\sigma_{\bullet}(A)}\cup \{0\},
$$
with $\sigma_{\bullet}$ denoting the point spectrum. The corresponding eigenvectors are identical.
\end{theorem}

Evidently, a function $f$ is an invariant density of $\op{t}$ for all $t\geq0$, if and only if $\gen f=0$. Further, since $\|\op{t}\|\leq 1$, the eigenvalues of $\gen$ lie in the left complex half-plane.  The family $\op{t}$ can be approximated by a truncated ``Taylor series'':
\begin{proposition}[\cite{BiKoJu14}]\label{truncatedtaylor}
If $f$ is $2N+2$ times continuously differentiable and $L^nf$, $n=0,1,\ldots,N$, is square-integrable with respect to~$\mu$, then 
$$
\Big\|\op{t}f-\sum_{n=0}^N\frac{t^n}{n!}\gen^{n}f\Big\|_{\mathcal{L}^2_\mu} = \mathcal{O}(t^{N+1}) ~~ \text{for $t\rightarrow 0$},
$$
where $\mathcal{L}^2_\mu$ denotes the $\mathcal{L}^2$-norm with respect to~$\mu$.
\end{proposition}

\section{Spatial dynamics and metastability} \label{sec:SpatialDynamicsMetastability}

{  
Consider an infinite number of systems modeled by (\ref{langevindynamics}) in thermodynamic equilibrium, i.e.\ identically and independently distributed according to $\fcan$ (this collection of systems is called an \emph{ensemble} in statistical mechanics). We are now interested in the portion of these systems which undergo a certain configurational change, i.e.\ leave a subset $A\subset\Qspace$ and enter another subset~$B\subset\Qspace$. For this, we track the evolution of~$\chi_A\fcan$, which is given by $P^t(\chi_A\fcan)$. This will be the starting point of the subsequent analysis. }

\paragraph{The spatial transfer operator. }
Since we are only interested in the distribution of their configurations at time $t\ge 0$, we compute the marginal with respect to $q$.  The resulting \emph{spatial transfer operator} for some~$u\in\mathcal{L}^2_{f_{\Qspace}}$ is \cite{Sch99,Web12}
\begin{equation} \label{spatialto}
S^tu(q):=\frac{1}{\fQ(q)}\int_\Pspace\opl{t}\big(u(q)\fcan(q,p)\big)~dp.
\end{equation}

\paragraph{Metastability on configuration space.} Using the scalar product
$$
\langle u,v\rangle_{\fQ}:=\int_\Qspace u(q)v(q)\fQ(q) dq
$$
(which gives rise to the norm~$\|\cdot\|_{\mathcal{L}^2_{f_{\Qspace}}}$), and the ``slice'' $\Gamma(A) :=\big\{(q,p)\in\Qspace~|~q\in A\big\}$ in state space, we define transition probabilities between slices via
\begin{equation}
\label{slicetransitionprob}
 p\left(t,\Gamma(A),\Gamma(B)\right)=\frac{\langle S^t\chi_A,\chi_B\rangle_{\fQ}} {\langle\chi_A,\chi_A\rangle_{\fQ}}.
\end{equation}
We call a disjoint union $A_1\cup\ldots\cup A_n=\Qspace$ of position space \emph{metastable} or \emph{almost invariant} if
$$
p\left(t,\Gamma(A_j),\Gamma(A_j)\right)\approx 1,~j=1,\ldots,n
$$
for the time scales~$t>0$ of interest. Other, more sophisticated, notions of metastability can be found in~\cite{SchSa13}.

The link between almost invariant/metastable sets and eigenvalues close to one and the corresponding eigenvectors of some transfer operator was first established in \cite{Del99} and used for conformation dynamics in \cite{Deu96}.  We here cite an extension to a broader class of transfer operators from \cite{Hui05}. 

\begin{theorem}[Application of \cite{Hui05}, Theorem 2]\label{metastability}
Let $\sigma(S^t)\subset [a,1]$ with $a>-1$ and \mbox{$\lambda_n\leq\ldots\leq\lambda_2<\lambda_1=1$} be the $n$ largest eigenvalues of $S^t$, with eigenvectors $v_n,\ldots,v_1$.
Let $\{A_1,\ldots,A_n\}$ be a measurable partition of $\Qspace$ and $\Pi$ be the orthogonal projection onto $\operatorname{span}(\chi_{A_1},\ldots,\chi_{A_n})$.  Then
$$
1+\rho_2\lambda_2 +\cdots+\rho_n\lambda_n + c \leq p(t,\Gamma(A_1),\Gamma(A_1))+\cdots + p(t,\Gamma(A_n),\Gamma(A_n)) \leq 1+\lambda_2+\cdots+\lambda_n,
$$
where $\rho_j=\| \Pi v_j\|\in [0,1]$ and $c= a(1-\rho_2+\ldots+1-\rho_n)$.
\end{theorem}

{  
\begin{remark}
We briefly discuss some specific properties of the spatial dynamics that are useful to understand the concept of pseudo generators outlined below.   
\begin{enumerate}
\item The spatial transfer operator~$S^t$ from~\eqref{spatialto} satisfies the assumptions in Theorem~\ref{metastability}; see~\cite[Appendix~B]{BiKoJu14}. Unfortunately, $S^t$ lacks the semi-group property and so cannot be the solution operator of an autonomous transport equation like the Fokker--Planck equation. Equivalently, spatial dynamics is not induced by an It\^o diffusion process and thus has no infinitesimal generator in the sense of \eqref{fokkerplanckequationito}.
\item The closer the eigenvalues $\lambda_2,\ldots,\lambda_n$ are to 1, the more metastable can a partition potentially be (upper bound in the theorem). How metastable a \emph{given} partition is, is controlled by the~$\rho_i$ (lower bound), which measure the constancy of the eigenfunctions on the partition elements. The better the eigenfunctions can be approximated by piecewise constant functions over the partition, the closer the~$\rho_i$  are to~1, and the more metastability is guaranteed by the lower bound. Also, note that since~$S^t$ is not a semigroup, the eigenfunctions~$v_i$ depend on~$t$ (cf.~Figure~\ref{fig:eigenfunctions} below). As a consequence, metastability of a partition must be understood \wrt the characteristic time scale $t>0$; 
\item It is not necessary that the sets~$A_i$ form a full partition, i.e.~$\bigcup_i A_i = \Qspace$. Similar results to the above have been obtained for non-complete partitions, where the~$A_i$ are considered to be \emph{cores of the metastable sets}; cf.~\cite[Section~5]{SchSa13}. 
\end{enumerate}
\end{remark}}

\subsection{Pseudo generators}
\label{ssec:pseudogen}
{  
Even though the spatial dynamics is lacking the semigroup property, it is possible---at least formally and in analogy with (\ref{generator})---to differentiate $S^t$ at $t=0$. We will see in the following that the resulting operators can play the role of the infinitesimal generator in the context of metastability analysis.}
 
\begin{definition}\label{pseudogenerator}
Let $\mathcal{X}$ be a Banach space, $T^t:\mathcal{X}\rightarrow \mathcal{X},~ t>0$ be a time-parametrized family of bounded linear operators. The operator 
$$
\frac{d}{dt}T^tf = \underset{h\rightarrow 0}{\lim}\frac{T^{t+h}f-T^tf}{h}
$$
is called the \emph{time-derivative} of $T^t$.  Iteratively, we define by
$\frac{d^n}{dt^n}T^t:=\frac{d}{dt}\big(\frac{d^{n-1}}{dt^{n-1}}T^t\big)$
the $n$-th time-derivative.  The operator 
$$
\pgen{n} := \frac{d^n}{dt^n} T^t\big|_{t=0}
$$
is called the $n$-th \emph{pseudo generator} of $T^t$.
\end{definition}
For $T^t=\op{t}$, the transfer operator of an It\^o process, the pseudo generators are simply
$\pgen{n}=\gen^{n}$, where $\gen$ is the infinitesimal (forward) generator.  

The pseudo generators of the spatial transfer operator $S^t$ can be expressed by the generator $\genl$ of the full Langevin transfer operator:
\begin{lemma}[\cite{BiKoJu14}]\label{altpseudogen}
The $n$-th pseudo generator $\pgen{n}$ of $S^t$ takes the form
$$\pgen{n}u(q)=\frac{1}{\fQ(q)}\int_\Pspace\genl^n\big(u(q)\fcan(q,p)\big)~dp.$$
Explicitly, we have
\begin{enumerate}[(1)]
\item $\displaystyle \pgen{1}=0$,
\item $\displaystyle \pgen{2}=\frac{1}{\beta}\Delta -\nabla V\cdot\nabla$. In particular, $\pgen{2}$ is independent of $\gamma$.
\end{enumerate} 
\end{lemma}

Surprisingly, one has 
\begin{corollary}[\cite{BiKoJu14}] \label{g2smolu}
The pseudogenerator $\pgen{2}$ (of the spatial transfer operator) is the infinitesimal generator of the Smoluchowski dynamics:
$$
\pgen{2}=G_\text{Smol}.
$$
\end{corollary}

\begin{remark}
Note that~$G_2 = G_{\rm Smol}$ has the form of the backward Smoluchowski generator~$A_{\rm Smol}$ (cf.\ Section~\ref{sec:largetime}). Still, $G_2$ is also the forward generator of the Smoluchowski process, if distributions are thought of as distributions \wrt the Gibbs--Boltzmann density~$f_{\Qspace}$. This is in accordance with the definition of the spatial transfer operator~\eqref{spatialto}, which also describes redistribution of mass with respect to~$f_{\Qspace}$. The formal coincidence ``$G_{\rm Smol} = A_{\rm Smol}$'' is not accidentally, but rather it reflects the reversibility of the Smoluchowski process.
\end{remark}

\paragraph{Taylor reconstruction of the spatial transfer operator.}

It is natural to ask whether there is an analogue of Proposition \ref{truncatedtaylor} for $S^t$ and its pseudo generators. We have the following result:
\begin{theorem}[\cite{BiKoJu14}]\label{spatialtaylor}
If $u$ is sufficiently regular, then 
$$
\left\|S^tu - \sum_{k=0}^K\frac{t^k}{k!}\pgen{k}u\right\|_{\mathcal{L}^2_{f_{\Qspace}}} = \mathcal{O}(t^{K+1}), \quad (t\rightarrow 0).
$$
\end{theorem}

Unfortunately, for $k>3$, higher derivatives of the potential~$V$ appear in the expressions for $\pgen{k}$, which are therefore impractical to work with, while the gradient~$\nabla V$  is typically available. We call
\begin{equation}\label{taylorrestoredoperator}
\begin{aligned}
R^tu &:= \Big(\operatorname{id} + \frac{t^2}{2}\pgen{2}\Big)u
= u + \frac{t^2}{2}\Big(\frac{1}{\beta}\Delta u -\nabla u\cdot \nabla V\Big)
\end{aligned}
\end{equation}
the \emph{(2nd order) Taylor approximation} of~$S^t$ such that if~$u$ is sufficiently regular,
$$\big\| S^tu - R^tu\big\|_{\mathcal{L}^2_{f_{\Qspace}}} = \mathcal{O}(t^3),\quad (t\rightarrow 0).$$

\paragraph{Exponential reconstruction.}

Unfortunately, unlike $S^t$, $R^t$ is not norm-preserving for densities with respect to $\fcan$. Therefore, when transporting $u$, we lose the interpretation of $\left(R^tu\right)\fcan$ as a physical density.  Moreover, $R^tu$ is not even bounded in $t$ \cite{BiKoJu14}. This quickly (i.e. for small $t$) destroys the interpretation of the eigenvalues of $R^t$ as a measure of metastability.

{  
An alternative approximation to $S^t$ preserves those properties. The Taylor approximation~\eqref{taylorrestoredoperator} already suggests that~$S^t$ behaves similarly to a~$\tfrac12 t^2$-scaled Smoluchowski dynamics. Hence, we define
\[
E^tf := P^{t^2/2}_{\rm Smol}f,
\]
where $P^t_{\rm Smol}$ is the semigroup of transfer operators generated by~$G_2$. This operator is integral-preserving with respect to the weight~$f_{\Qspace}$~\cite{BiKoJu14}, and we get the following analogue for Proposition~\ref{truncatedtaylor}:}

\begin{lemma}[{\cite[Lemma~4.10]{BiKoJu14}}]\label{experror}
If $u$ is sufficiently regular, then for $t\rightarrow 0$,
$$
\bigg\|E^tu - \sum_{n=0}^N \frac{\big(\frac{t^2}{2}\pgen{2}\big)^n}{n!}u\bigg\|_{\mathcal{L}^2_{f_{\Qspace}}} = \mathcal{O}(t^{2N+1}).
$$
In particular,
$$
\big\|E^tu - S^tu\big\|_{\mathcal{L}^2_{f_{\Qspace}}} = \mathcal{O}(t^3) \quad (t\to 0).
$$
\end{lemma}

\paragraph{Reconstruction of eigenspaces.}
The error asymptotics carries over to the spectrum and eigenvectors of $S^t$, $R^t$ and $E^t$ in the following way:

\begin{corollary}[\cite{BiKoJu14}] \label{taylorspectralerror}
Let $u$ be a sufficiently regular eigenvector of~$R^t$ or of~$E^t$ to eigenvalue~$\lambda$. Then
\[
\|S^tu-\lambda u\|_{\mathcal{L}^2_{f_{\Qspace}}} = \mathcal{O}(t^3).
\]
\end{corollary}
Thus, for small~$t$ we may interpret dominant eigenpairs~$(u,\lambda)$ of~$E^t$ and~$R^t$ as good approximations to dominant eigenpairs of~$S^t$. Hence, they can be used to define metastable sets following the spatial decomposition approach in~\cite{DeWe03}. The eigenfunctions of interest, those of~$S^t$, $E^t$, and~$R^t$, can be shown to be sufficiently regular under fairly general conditions, cf.~\cite[Appendix~C]{BiKoJu14}.

{  
\begin{remark}
Corollary \ref{taylorspectralerror} is in accordance with functional analytical results by Nier and co-workers (e.g.~\cite{HeNi04}) that show that the dominant spectrum of the non-reversible Langevin dynamics is real-valued and close to the spectrum of the reversible Smoluchowski dynamics, even for moderate values of the friction coefficient. In \cite{HeNi04}, this somewhat surprising result is obtained by large deviations arguments for the small-noise limit  using the Witten Laplacian representation of the (hypoelliptic) Langevin generator, whereas the considerations here  and in \cite{BiKoJu14} are based on small-time asymptotics of the spatial Langevin dynamics. We believe that a connection between these results is that the small-noise limit can be understood as an exponential rescaling of time as is suggested by large deviations theory; cf.~\cite{PeEtal10}. We refrain from going into details here and leave the analysis of this interesting connection to future work.      
\end{remark}}

\subsection{Towards spatial generators in essential coordinates}

We have discussed the concept of the spatial transfer operator that is obtained by projecting the phase space dynamics onto the spatial components. We shall now consider the restriction of the dynamics to a given collective variable, also termed \emph{essential coordinate}. To this end, let $\xi\colon\Qspace\to \Zspace\subset\R$ be a smooth map with the property that, for every regular value $z\in\Zspace$ of $\xi$, the level sets 
\[
\Mspace_{z}=\left\{q\in\Qspace\colon \xi(q)=z\right\}\subset\Qspace
\]
are smooth submanifolds of $\Qspace$ with codimension 1 (i.e.~hypersurfaces). We suppose that $\xi$ is a physically relevant observable of the dynamics, such as a reaction coordinate or some collective variable that monitors a conformational transition, and call $\xi$ the \emph{essential coordinate}; the unessential coordinates are then implicitly defined through the foliation of $\Qspace$ by the map $\xi$, in other words: the unessential coordinates parameterize the leaves $\Mspace_{z}$ of the foliation for every (regular) value $z$ of $\xi$. 

To define the analogue of the spatial transfer operator (\ref{spatialto}) for the essential coordinate, firstly note that~\cite[Section~3.2]{Fed69}
\begin{equation}\label{coarea}
\int_{\Qspace} g(q)\,dq = \int_{\Zspace}\left(\int_{\Mspace_{z}} g |\nabla\xi|^{-1}d\sigma_{z}\right)dz
\end{equation}
for any integrable function $g\colon\Qspace\to\R$ where $d\sigma_{z}$ denotes the Riemannian volume element on $\Mspace_{z}$. Equation (\ref{coarea}) is called the \emph{coarea formula} and can be considered a nonlinear variant of Fubini's theorem.  

Together with the law of total expectation, the coarea formula thus entails that the canonical probability measure $\mu$ conditional on $\xi(q)=z$ has the form
\begin{equation}\label{condprob1}
d\mu_{z} = \frac{1}{N(z) } \fcan|\nabla\xi|^{-1} d\sigma_{z}dp \,,
\end{equation}
with the normalization constant
\begin{equation}\label{condprob2}
N(z) = \int_{\Mspace_{z}\times\Pspace} \fcan |\nabla\xi|^{-1}d\sigma_{z}dp \,.
\end{equation}
The spatial transfer operator $S^{t}$ for essential coordinates can now be defined as 
\begin{equation}\label{essto}
S_{\rm ess}^t w(z):=\frac{1}{N(z)}\int_{\Mspace_{z}\times\Pspace}\opl{t}\big(w(\xi(q))\fcan(q,p)\big) |\nabla\xi(q)|^{-1} d\sigma_{z}dp\,,
\end{equation}

\paragraph{Projected pseudo-generators.}

To compute the corresponding pseudo-generators, let $\rho$ be the configurational marginal probability measure that is obtained by projecting $\mu$ onto the configurations by integrating out the momenta, i.e., $d\rho(q)=f_{\Qspace}(q)dq$. Let us further introduce a projection operator $\Pi_{z}\colon \mathcal{L}^{2}_{\rho}(\Qspace)\to \mathcal{L}^{2}_{\rho}(\Qspace)$ by  
\begin{equation}
(\Pi_{z} u)(z) =  \frac{1}{N_{\Qspace}(z)}\int_{\Qspace}u(q) \,f_{\Qspace}(q)|\nabla\xi(q)|^{-1} d\sigma_{z}(q)
\end{equation}
where $f_{\Qspace}$ is the $q$-marginal of $\fcan$ and $N_{\Qspace}$ is the corresponding normalization constant for the conditional density. It can be readily seen that, $\Pi_{z}$ is an orthogonal projection \wrt the natural (weighted) scalar product in the space $\mathcal{L}^{2}_{\rho}(\Qspace)$ and amounts to the expectation of functions \wrt $\rho$ conditional on $\xi(q)=z$. 

Thus, for functions $u(q)=w(\xi(q))$, the reduced spatial transfer operator $S_{\rm ess}^{t}$ and the spatial transfer operator (\ref{spatialto}) are related by (cf.~\cite{Sch99})
\begin{equation}\label{spatialessto}
S_{\rm ess}^t w(z) = (\Pi_{z} S^t u)(z)\,.
\end{equation}
The last identity is helpful in computing the corresponding pseudo generators $G_{n}^{\rm ess}$. Here we are interested only in the second pseudo generator $G_{2}^{\rm ess}$, for which we have the following analogue of Lemma \ref{altpseudogen}: 
\begin{lemma} \label{lemma:projectedpseudogens}
For sufficiently smooth functions $u(q)=w(\xi(q))$, the $n$-th pseudo generator of $S_{\rm ess}^t$ reads
\[
G_{n}^{\rm ess}w(z) =  (\Pi_{z} G_{n}u)(z)
\]
Specifically, we have 
\[
G_{2}^{\rm ess} = \beta^{-1}a(z)\frac{\partial^{2}}{\partial z^{2}} + b(z)\frac{\partial}{\partial z}\,,
\]
with the noise and drift coefficients
\[
a(z) = (\Pi_{z}|\nabla\xi|^{2})(z)\,,\quad b(z) = \left(\Pi_{z}(\beta^{-1}\Delta\xi - \nabla\xi\cdot \nabla V)\right)(z)\,.
\]
\end{lemma}

\begin{proof}
The first part of the assertion is a straight consequence of Lemma \ref{altpseudogen} and the coarea formula. As for the second part, observe that the second pseudo generator is given by $G_{2}=\beta^{-1}\Delta-\nabla V\cdot\nabla$ which by chain rule implies:  
\[
G_{2}w(\xi(q)) =  \beta^{-1}|\nabla\xi|^{2}w''(z)|_{z=\xi(q)} + (\beta^{-1}\Delta\xi - \nabla\xi\cdot \nabla V) w'(z)|_{z=\xi(q)}\,.
\]  
Letting the projection $\Pi_{z}$ act from the left using that $\Pi_{z}w'(z)|_{z=\xi(q)}=w'(z)$ and likewise $\Pi_{z}w''(z)|_{z=\xi(q)}=w''(z)$ gives the desired result. 
\end{proof}
A few remarks are in order: 
\begin{enumerate} 

\item In accordance with Corollary \ref{g2smolu}, the second projected pseudo generator $G_{2}^{\rm ess}$ is the infinitesimal generator of the diffusion
\begin{equation}\label{essSmolu}
\frac{dz}{dt} = b(z) + \sqrt{2\beta^{-1}}\sigma(z)\zeta_t\,,
\end{equation}
with $\sigma(z)=\sqrt{a(z)}$ and $\zeta_t$ being a one-dimensional uncorrelated Gaussian white noise process. Equation (\ref{essSmolu}) has been derived by Legoll and Leli\`evre \cite{LeLe10} using first-order (Markovian) optimal prediction.\\

\item In \cite{LeLe10}, the authors prove an error bound for the projected dynamics (\ref{essSmolu}) under the assumption that the conditional probability $\mu_{z}$ satisfies a logarithmic Sobolev inequality. {  We refrain from transferring the analysis to our situation as logarithmic Sobolev constants are difficult to estimate (beyond the case of strictly convex potentials or in the zero-noise limit), hence the approach is of limited practical use.} Nonetheless, we believe that the projected pseudo generator $G_{2}^{\rm ess}$ will provide a good approximation of the dominant spectrum of $L_{\text{Lan}}$ whenever $\xi$ is a slow coordinate relative to the unessential configuration variables and the momenta. \\

\item If $|\nabla\xi|$ is bounded above and away from zero, it can be shown (see \cite{LeLe10}) that the process $(z_{t})_{t\ge 0}$ generated by $L_{\rm ess}=G_{2}^{\rm ess}$ has the unique invariant measure 
\[
d\nu(z) = \exp(-\beta  F(z))\,dz
\]
with
\[
F(z) = -\beta^{-1}\log \int f_{\Qspace}|\nabla\xi|^{-1}d\sigma_{z}
\]
being the thermodynamic free energy in the essential coordinate. Note that $\nu=\mu\circ\xi^{-1}$ is the push-forward of the canonical distribution by $\xi$ (i.e.~the $\xi$-marginal). Naively, one might expect the projected Smoluchwski equation to be of the form
\begin{equation}\label{essSmolu2}
\frac{dy}{dt} = - F'(y) + \sqrt{2\beta^{-1}}\zeta_t\,,
\end{equation}
and it can be shown that (\ref{essSmolu}) can be transformed into (\ref{essSmolu2}) according to $y=\varphi(z)$ using It\^o's Lemma with $\varphi$ being  the volatility transform
\[
\varphi(z) = \int_{0}^{z} (\sigma(s))^{-1}ds\,
\]
that leads to a Smoluchowski equation with unit noise coefficient \cite{EVE04}. { As a consequence, (\ref{essSmolu}) can be equivalently expressed as 
\begin{equation}\label{essSmolu3}
\frac{dz}{dt} = -a(z) F'(z) +\beta^{-1}a'(z) +  \sqrt{2\beta^{-1}}\sigma(z)\zeta_t\,,
\end{equation}
which is exactly the one-dimensional analogue of (\ref{smoluchowskidynamicsGen})--(\ref{ito}) with $\gamma=a^{-1}$.}\\

\item In order to use $G_2^{\rm ess}$ in metastability analysis (analogous to $G_2$ in section \ref{ssec:pseudogen}), it has to be discretized. The method of choice is spectral collocation due to the regularity of the objects of interest~\cite{BiKoJu14} (i.e. eigenfunctions of $S^t_{\rm ess}$). Here, collocation requires the evaluation of $G_2^{\rm ess}\phi_i(z_j)$ for ansatz functions $\phi_i$ at collocations points $z_j$ (see Section \ref{discussion} for details). This in turn requires the evaluation of the noise and drift-coefficients $a(z_j),~b(z_j)$ in Lemma \ref{lemma:projectedpseudogens}, which involve (potentially high-dimensional) integrals that represent averages over the non-essential degrees of freedom; see, e.g., \cite{CiLeVE08,Har08,FEBook} for Monte-Carlo methods to efficiently compute these high-dimensional integrals. 

{ 
\item We should stress that Lemma \ref{lemma:projectedpseudogens} can be readily generalized to multidimensional reaction coordinates, however, in general (expect for the case of pairwise orthogonal reaction coordinates) it is unclear whether the physical interpretation of the projected equation as a reversible diffusion in the free energy landscape is retained. 
}

\end{enumerate}

\section{Approximation quality for larger time scales}	\label{sec:largetime}

We have seen in Section~\ref{ssec:pseudogen} that $E^t = P^{t^2/2}_{\rm Smol}$ approximates $S^t$ well (pointwise) for small times~$t$. However, for metastability analysis, spectral properties of the spatial operator for larger time scales are of interest. In this section we make use of two well-known techniques---\emph{perturbation expansion}, already seen in Section~\ref{ssec:models}, and the \emph{Mori--Zwanzig formalism}---with the aim of explaining the approximation quality of pseudo generator reconstructions of~$S^t$, and extending them to larger time scales. Then, we discuss how to utilize the ergodicity of the Langevin process to show an almost Markovian behaviour of the spatial dynamics on long time scales. This eventually leads to a bound on the time scale on which the spatial dynamics is well approximated.

\subsection{Perturbation expansion}	\label{ssec:perturbation}

The idea of perturbation expansion rests on the assumption that there exists a small problem parameter in which one can expand the objects of interest in a (formal) power series. As in Section~\ref{ssec:models}, here this small parameter is the inverse of the damping in the Langevin dynamics, i.e.\ $\ep :=\gamma^{-1}$ where, for simplicity, we assume that the friction coefficient is scalar. For ease of presentation, we set the inverse temperature $\beta=1$.

It turns out to be advantageous to work with the propagators (Koopman operators) instead of the transfer operators themselves. The difference is only of technical nature, since the propagators are the adjoints of the corresponding transfer operators. Denoting the propagators of the Langevin, Smoluchowski, and spatial dynamics by $T^t_{\rm Lan}$, $T^t_{\rm Smol}$, and $T^t_{\Qspace}$, respectively, we have the explicit representations
\begin{eqnarray}
T^t_{\rm Lan}u(q,p) & = & \mathbb{E}\left[u(\bm q_t^{\rm Lan},\bm p_t^{\rm Lan})\,\big\vert\, \bm q_0^{\rm Lan}=q,\ \bm p_0^{\rm Lan}=p\right]\\
T^t_{\rm Smol}w(q) & = & \mathbb{E}\left[w(\bm q_t^{\rm Smol})\,\big\vert\, \bm q_0^{\rm Smol}=q\right]\\
T^t_{\Qspace}w(q) & = & \int_{\Pspace}\mathbb{E}\left[w(\bm q_t^{\rm Lan})\,\big\vert\, \bm q_0^{\rm Lan}=q,\ \bm p_0^{\rm Lan}=p\right]f_{\Pspace}(p)dp
\end{eqnarray}
where the expectation~$\mathbb{E}[\cdot]$ is taken with respect to the law of the stochastic forcing in the Langevin (for $T^t_{\rm Lan}$ and $T^t_{\Qspace}$) and Smoluchowski (for $T^t_{\rm Smol}$) equations. The propagators $T^t_{\rm Lan}$ and $T^t_{\rm Smol}$ are semigroups with generators
\begin{eqnarray*}
A_{\rm Smol} & = & \Delta_q - \nabla_q V\cdot \nabla_q\\
A_{\rm Lan} & = & p\cdot\nabla_q - \nabla_q V\cdot\nabla_p + \ep^{-1}\left(\Delta_p - p\cdot\nabla_p\right) = A_{\rm Ham} + \ep^{-1}A_{\rm OU}
\end{eqnarray*}
while $T^t_{\Qspace}$ is not a semigroup, but $\tfrac{d^2}{dt^2}T^t_{\Qspace}\big\vert_{t=0} = A_{\rm Smol}$; in complete analogy with the theory presented above.

To proceed, set $A_{\ep} := \ep^{-1} A_{\rm Lan}$. This scaling of $A_{\rm Lan}$ is called \emph{diffusive scaling} and is due to the fact that the spatial dynamics gets slower and slower when friction is increased, and nontrivial dynamics only takes place on time scales of order $\ep^{-1}$. The scaling of~$A_{\rm Lan}$ by~$\ep^{-1}$ thus restores the relevant dynamics; see also (\ref{convgHFL}). 

Now let $(\lambda_{\ep},u_{\ep})$ be an eigenpair of $A_{\ep}$, such that $A_{\ep}u_{\ep} = \lambda_{\ep}u_{\ep}$, and assume the existence of formal series expansions
\begin{eqnarray*}
u_{\ep} & = & u_0 + \ep u_1 + \ep^2 u_2 + \ldots \\
\lambda_{\ep} & = & \lambda_0 + \ep \lambda_1 + \ep^2 \lambda_2 + \ldots
\end{eqnarray*}
It follows (see, e.g., \cite[pp. 43]{SchSa13}, \cite{Pav08}) that~$u_0(q,p) = u_0(q)$, with~$A_{\rm Smol}u_0 = \lambda_0 u_0$, and~$u_1(q,p) = {p\cdot \nabla_q u_0(q)}$.
This already gives a formal justification of the Smoluchowski dynamics as overdamped limit of the Langevin dynamics: on a time scale $\tau = \ep t$ (recall that $A_{\rm Lan} = \ep A_{\ep}$) the position coordinate of the Langevin dynamics is governed by the Smoluchowski dynamics, up to fluctuations of order $\ep$.\\

A closer look at the structure of the first terms in the eigenfunction expansion reveals even more. Metastability information is contained in eigenfuntions at nonzero eigenvalues, hence let $\lambda_0 \neq 0\neq \lambda_{\ep}$. Since $A_{\rm Smol}$ is the formal adjoint of $L_{\rm Smol}$ and $A_{\rm Lan}$ is the formal adjoint of $L_{\rm Lan}$ in~$\mathcal{L}^{2}$, their eigenfunctions to different eigenvalues are orthogonal with respect to the corresponding scalar product. And, since the eigenfunctions of $L_{\rm Lan}$ and $L_{\rm Smol}$ at the eigenvalue 0 are the canonical and Gibbs--Boltzmann densities~$\fcan$ and~$f_{\Qspace}$, respectively, we have that
\[
\int_{\Qspace}\int_{\Pspace} f_{\rm can}(q,p)u_{\ep}(q,p)dpdq = 0\,,
\]
and
\[
\int_{\Qspace}\int_{\Pspace} f_{\rm can}(q,p)u_0(q,p)dpdq = \int_{\Qspace} f_{\Qspace}(q)u_0(q)dq = 0\,.
\]
Being in the subspace orthogonal to $\fcan$, both functions decay exponentially under the action of the propagator. More precisely, let
\[
\eta_{\ep} := \max \left\{\mathrm{Re} \lambda_{\ep}\,\vert\, 0\neq\lambda_{\ep}\in\sigma(A_{\ep})\right\}<0
\]
be the real part of the nonzero eigenvalue of $A_{\ep}$ which is closest to zero; i.e.\ $(\ep|\eta_{\ep}|)^{-1}$ is the dominant time scale of the Langevin dynamics. Note that $\eta_{\ep} = \mathcal{O}(1)$ as $\ep\to 0$, and $\limsup_{\ep\to0}\eta_{\ep} < 0$. Now, both $T^t_{\rm Lan}u_{\ep}$ and $T^t_{\rm Lan} u_0$ decay as $\exp(\ep\eta_{\ep} t)$ for $t\to\infty$. We will utilize this with the perturbation expansion in the following computation. Its purpose is to estimate how far the Smoluchowski eigenfunction $u_0$ is from being an eigenfunction of the spatial propagator $T^t_{\Qspace}$.
\begin{eqnarray*}
T^t_{\Qspace}u_0(q) & = & \int_{\Pspace} \left(T^t_{\rm Lan}u_0\right)(q,p)f_{\Pspace}(p)dp \\
	& = & \int_{\Pspace} \left(T^t_{\rm Lan}\left(u_{\ep}-(u_{\ep}-u_0)\right)\right)(q,p)f_{\Pspace}(p)dp \\
	& = & e^{\ep\lambda_{\ep}t}\int_{\Pspace}u_{\ep}(q,p)f_{\Pspace}(p)dp + \mathcal{O}(e^{\ep\eta_{\ep} t}\ep) \\
	& = & e^{\ep\lambda_{\ep}t}\int_{\Pspace}\left(u_0(q)+\ep u_1(q,p) + \mathcal{O}(\ep^2)\right)f_{\Pspace}(p)dp + \mathcal{O}(e^{\ep\eta_{\ep} t}\ep) \\
	& = & e^{\ep\lambda_{\ep}t}u_0(q) + \mathcal{O}(e^{\ep\lambda_{\ep} t}\ep^2) + \mathcal{O}(e^{\ep\eta_{\ep} t}\ep)\quad\text{as }\ep\to 0\,,
\end{eqnarray*}
where the third equality is obtained by utilizing the exponential decay of ${T^t_{\rm Lan}(u_{\ep}-u_0)}$. The last equality follows from $u_1$ and $f_{\Pspace}$ being odd and even functions of~$p$, respectively, hence the integral of their product vanishes. On the new, slower time scale $\tau = \ep t$ we obtain
\[
T^{\ep^{-1}\tau}u_0 = e^{\lambda_0\tau + \mathcal{O}(\ep)}u_0 + e^{\lambda_0\tau + \mathcal{O}(\ep)}\mathcal{O}(\ep^2) + e^{\eta_{\ep}\tau}\mathcal{O}(\ep)\,.
\]
This means that $u_0$ is an approximate eigenfunction of the spatial propagator $T^{\ep^{-1}\tau}_{\Qspace}$ as long as $e^{\lambda_0\tau}$ dominates the last two terms on the right hand side.\footnote{The spatial transfer operator is self-adjoint, hence normal. From the theory of pseudospectra for normal operators~\cite{TrEm05} we know that if $Tu = \lambda u + \ep v$ for some linear operator~$T$, $u,v$ of modulus one, and $\lambda\in\mathbb{R}$, then~$T$ has an eigenvalue in the $\ep$-neighborhood of~$\lambda$.} 
It clearly dominates the second term (since we assume $\ep$ to be small), hence we arrive at the desired condition by comparing it with the third:\footnote{Note that for $x,y>0$ one has $e^{-x}\ll e^{-y}$ if $e^{y-x}\ll 1$, already achieved if $y\lesssim x$, meaning ``$y$ smaller than $x$ up to some additive constant''.}
\begin{equation}
\tau \lesssim \frac{1}{|\lambda_0-\eta_{\ep}|}|\log\ep|\qquad\text{or}\qquad t \lesssim \frac{1}{|\lambda_0-\eta_{\ep}|}\ep^{-1}|\log\ep|
\label{eq:lag_estimate}
\end{equation}
These estimates allow the following interpretation:
\begin{enumerate}[1.]
\item While the standard result allows an approximation of the (position coordinate of the) Langevin dynamics by the Smoluchowski dynamics on a time scale $\ep t = \tau = \mathcal{O}(1)$ (as $\ep\to 0$), our estimate suggests that with respect to metastability analysis this time scale can be stretched by a factor $|\log \ep|$.\\
\item The more dominant an eigenvalue, i.e.\ the smaller $|\lambda_0-\eta_{\ep}|$, the longer the time scale is on which the Smoluchowski eigenmode approximates the corresponding eigenmode of the spatial propagator well. For the first subdominant eigenmode, where $\lambda_{\ep} = \eta_{\ep}$, and hence~$\lambda_0 - \eta_{\ep} = \mathcal{O}(\ep)$, the estimate reads as $\tau \lesssim \ep^{-1}|\log \ep|$, or equivalently, $t \lesssim \ep^{-2}|\log \ep|$.\\
\end{enumerate}

In order to validate the estimate (\ref{eq:lag_estimate}) also numerically, we perform the following experiment: 
Consider the Langevin system induced by the one-dimensional periodic potential
$$
V(q) = 1 + 3\cos(2\pi q) + 3\cos^2(2\pi q) - \cos^3(2\pi q)
$$
with constant mass matrix $M=1$ at temperature~$\beta=1$.

\begin{figure}[h]
\centering
\includegraphics{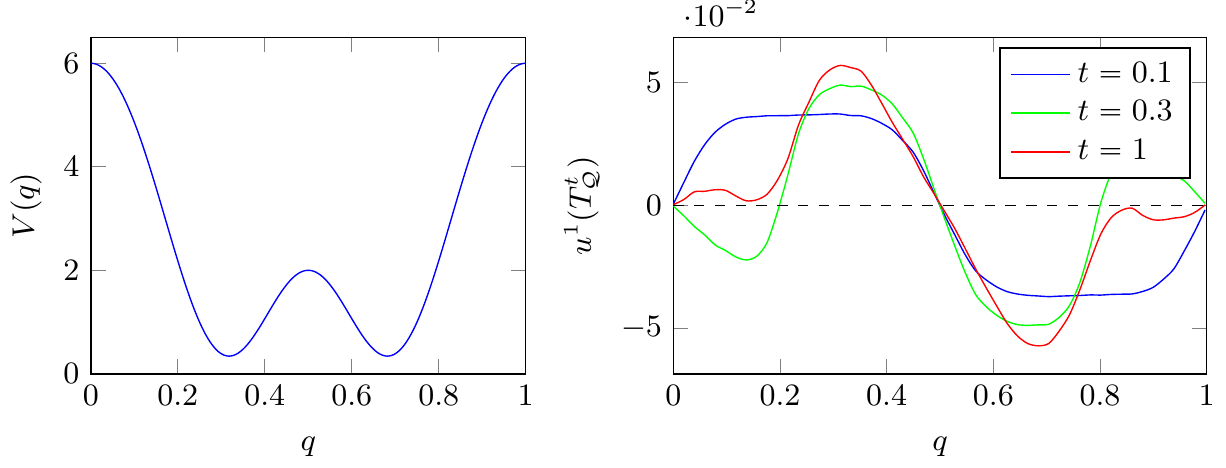}
\caption{The two wells of the periodic double well potential indicate two metastable regions in configuration space. The sign structure of the dominant eigenfunctions of $T^t_\Qspace$ reveals them.}
\label{fig:eigenfunctions}
\end{figure}

For varying $\ep=\gamma^{-1}$, we computed the largest lag time $t=t_\nu(\ep)$ such that the eigenfunctions $u_\ep=u_\ep^1$ at the subdominant eigenvalue $\lambda_\ep=\lambda^1_\ep$ of~$T^t_\Qspace$ and~$T^t_\text{Smol}$ differ by less than a given threshold $\nu$, i.e.\ we compute
\[
t_\nu(\ep) := 
\inf\left\{t > 0: \|u^1_\ep(T^t_\Qspace) - 
				   u^1_\ep(T^t_\text{Smol})\|_{\mathcal{L}^2_{f_{\Qspace}}} > \nu\right\}.
\]
Figure~\ref{fig:max_lag} shows $\ep\mapsto t_\nu(\ep)$ for $\nu=0.05$, and for comparison, the graph of $\ep\mapsto c_1\log(\ep)\ep^{-2}+c_2$ (where we obtained the constants~$c_1$ and~$c_2$ by a least squares fit on the given data).  Clearly, on the chosen domain for $t_\nu$, there is an excellent agreement with the estimate (\ref{eq:lag_estimate}). 

\begin{figure}[h!]
\centering
\includegraphics{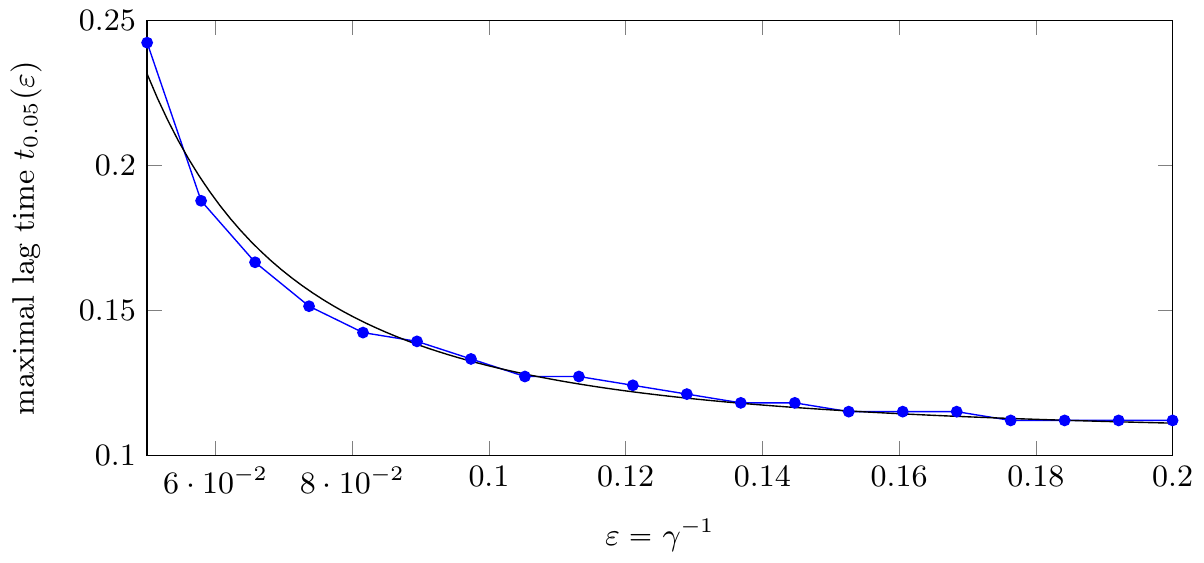}
\caption{$\ep$-dependence of the maximal lag time. The blue graph shows the largest lag time such that $\|u^1_\ep(T^t_\Qspace) - u^1_\ep(T^t_\text{Smol})\|_{\mathcal{L}^2_{f_{\Qspace}}}< 0.05$. The black graph is $c_1\log(\ep)\ep^{-2}+c_2$ with $c_1\approx -1.04\cdot 10^{-4},~c_2\approx 1.07\cdot 10^{-1}$ (from least-squares fitting). The eigenfunctions were computed using a simple Ulam discretization~\cite{BiKoJu14} of $T^t_\Qspace$ and $T^t_\text{Smol}$ with resolution $256$.}
\label{fig:max_lag}
\end{figure}


Although these first estimates allow merely a slight quantitative extension of the time scale on which the Smoluchowski dynamics approximates the spatial component of the  Langevin dynamics well, it suggests that the consideration of further structural information from the perturbation expansion may allow for an extension of approximation time scales beyond the current, or inspire corrections terms to do so.

\subsection{The Mori--Zwanzig formalism}	\label{ssec:MoriZwanzig}

In the previous section, we have analyzed a possibility to extend the time scales on which metastability information gained from the Smoluchowski equation is a good approximation to that of the actual model of interest, the Langevin dynamics.  The argument, however, required the smallness of the inverse damping coefficient $\ep = \gamma^{-1}$. In this section we turn to the question what can be done without this assumption.

The Mori--Zwanzig representation decomposes the differential equation governing the state variable of interest into terms according to their dependence on the same quantities of interest. Note that the formalism itself is quite general~\cite{Zwa73,ChHaKu00,ChHaKu02}; we will give an brief introduce that is tailored to our needs; cf.~also the related paper \cite{SkWo1979}. \\

Let $A$ be the infinitesimal generator of some propagator semigroup, which we will formally denote by $(e^{At})_{t\ge 0}$. This propagator acts on scalar functions $f$ which are functions of the \emph{full} state~$x$. Let $x = (\hat x,\tilde{x})$, where $\hat x$ is the state of interest (also called the \emph{resolved variables}). Let a distribution~$\mu$ be given over the state space, and define the projection operator~$\proj$ as expectation with respect to~$\mu$ conditional on~$\hat x$:
\[
\proj f(x) = \proj f(\hat x) := \mathbb{E}_{\mu}\left[f(x)\,\big\vert\, \hat x\right] = \frac{\int f(x)d\mu(\tilde{x})}{\int d\mu(\tilde{x})}
\] 
We are only interested in the evolution of average quantities conditional on~$\hat x$, i.e.\ in~$\proj e^{tA}$. Note that in the conformational analysis setting $x = (q,p)$,  $\hat x = q$, $\mu$ is the canonical measure with density $\fcan$, $A=A_{\rm Lan}$ from above, and thus $\proj e^{tA} = T^t_{\Qspace}$, the spatial propagator.

Let $\oproj = Id - \proj$ denote the projection orthogonal on the space of functions of~$\hat x$. A modified Mori--Zwanzig representation yields
\begin{equation}
\frac{d}{dt}\proj e^{tA} = \proj A \proj e^{tA} + \proj A \oproj \int_0^t e^{s A}\oproj Ae^{(t-s)\proj A}\,ds + \proj A \oproj e^{t\proj A}\,.
\label{eq:zwanzig1}
\end{equation}
It can be obtained by applying Dyson's formula~\cite{EvMo07} on~$e^{tA}$ in the second term on the right hand side of the identity $\tfrac{d}{dt}\proj e^{tA} = \proj A \proj e^{tA} + \proj A \oproj e^{tA}$, where we tacitly assume that the orthogonal dynamics in the space of the unresolved variables is well-posed; see  \cite{GiHaKu04} for details. 

Equation~\eqref{eq:zwanzig1} is hard to interpret in this form.  Again,  the structure of the problem at hand aids us: the assertions of the following Lemma can be checked by direct computation. From	 now on we work in the conformational analysis setting, i.e.~$\hat x$, $\mu$, and~$A$ are given as above.
\begin{lemma}	\label{lem:projids}
Let $f$ be a function independent of the variable $p$, i.e.\ $f(q,p) = f(q)$ $\forall q,p$. Then the following holds:
\begin{enumerate}[(a)]
\item $\proj f = f$ and $\oproj f = 0$;
\item $\proj A f=0$, thus $e^{t\proj A}f = f$ $\forall t\ge 0$, and $\proj A\proj = 0$;
\item $\proj A\oproj A f = A_{\rm Smol}f$.
\end{enumerate}
\end{lemma}
Applying identity~\eqref{eq:zwanzig1} to some~$f$ being a function only of the spatial variable~$q$, we obtain with Lemma~\ref{lem:projids}:
\begin{equation}
\frac{d}{dt}T^t_{\Qspace}f = \proj A \oproj \int_0^t e^{s A}\oproj Af\,ds\,.
\label{eq:zwanzig2}
\end{equation}
Approximating the integral as $\int_0^t h(s)ds = t\,h(0)+\mathcal{O}(t^2)$, we get
\[
\frac{d}{dt}T^t_{\Qspace}f = t \proj A \oproj Af + \mathcal{O}(t^2) = t A_{\rm Smol} f + \mathcal{O}(t^2)\,. 
\]
Integrating over $t$ yields
\begin{equation}
T^t_{\Qspace} f = f + \frac{t^2}{2}A_{\rm Smol} f + \mathcal{O}(t^3)\,.
\label{eq:reconstr1}
\end{equation}
This result can be seen as an analogue to~\eqref{taylorrestoredoperator}, with the advantage that it has been derived from an \emph{identity},~\eqref{eq:zwanzig2}, which now offers the possibility of a systematic exploitation of quadrature rules of increasing order to approximate the integral on its right hand side.
To this end, note that the range of $\oproj$ is orthogonal to $\fcan$. Thus, by the considerations in the previous section, $e^{tA}\oproj u = T^t_{\rm Lan}\oproj u$ decays exponentially as $t\to\infty$. This means that for some unknown $\lambda < 0$ the integral in~\eqref{eq:zwanzig2} has the form $\int_0^t e^{-\lambda s}g(s)e^{\lambda s}ds$, where $g(s)e^{-\lambda s} = \mathcal{O}(1)$ as $s\to\infty$. Integrals with exponential weights can be approximated by the Gau{\ss}--Laguerre quadrature rule. 

There are two issues which have to be addressed: the unknown~$\lambda <0$ and that we do not have an explicit expression for the function $g(s) = T^s_{\rm Lan}\oproj u$ for $s>0$. For the former, one could use the eigenvalues of the Smoluchowski propagator as an initial guess, and try to refine this estimate subsequently. For the latter, bootstrapping techniques (just as used for the derivation of Runge--Kutta methods in numerical integration) could help us to approximate $g(s)$ from derivatives of $g$ at~$s=0$.\\

We shall summarize which lines of attack so far the Mori--Zwanzig formalism offers to extend the time scales of approximation in molecular conformation analysis. 
\begin{enumerate}[1.]
\item Based on~\eqref{eq:zwanzig2} we are able to state a second-order accurate approximation of the spatial propagator by the Smoluchowski propagator. Using higher order quadrature rules (either based on Taylor expansions of the integrand or Gau{\ss}--Laguerre quadrature), it should be investigated whether \emph{simple higher order} approximations can be derived as well.\\
\item Repeating the Mori--Zwanzig procedure starting from the decomposition 
\[
\tfrac{d}{dt}\proj e^{tA} =  \proj e^{tA} \proj A + \proj e^{tA} \oproj A\,,
 \]
 one arrives at different representations than~\eqref{eq:zwanzig1}. This is the usual approach (cf.~\cite{ChHaKu02}), as it allows the interpretation of the arising terms as ``first order optimal prediction'', ``memory'', and ``noise'', however does not lead as simply to~\eqref{eq:reconstr1} as~\eqref{eq:zwanzig1} does. Nevertheless, it should be considered parallel to~\eqref{eq:zwanzig1}, as it may  reveal other important characteristics of the spatial dynamics.\\
 
\item In~\cite{ChHaKu02}, the short memory approximation $\int_0^th(s)ds \approx t\,h(0)$ has been used, however not for the position-momentum decomposition that we consider here. In the same work, different projections $\proj$ have been considered, e.g.\ finite rank projections to a set of basis functions. The results of Section~\ref{ssec:perturbation} suggest, that a projection of the spatial dynamics to the space spanned by dominant eigenfunctions of the Smoluchowski propagator may give a good approximation for larger times as well.
\end{enumerate}

\subsection{Almost Markovian behaviour: on bounding the approximation time scales}	\label{ssec:almostMarkov}

For small $t$, the non-Markovianity of spatial dynamics is an important feature which characterizes the density transport and metastability. We have seen that a~$\lambda\in\sigma(S^t)$ satisfies $\lambda\rightarrow 1$ as $t\to 0$ with a rate of $\mathcal{O}(\exp{(-\kappa t^2/2)})$ (for some suitable~$\kappa$), in contrast to the rate for semigroups of operators, which is~$\mathcal{O}(\exp{(-\kappa t)})$.

However, for larger $t$, $S^t$ exhibits a more regular, almost Markovian behaviour~\cite{ChodEtAl06,SwPiSu04}. We give ideas as to how this could be exploited for efficient computation of the eigenvalues of~$S^t$ in this time region. 

\paragraph{Relaxation times for momenta distributions.}
Langevin dynamics, the underlying model of $S^t$, is both Markovian and ergodic~\cite{MaSt02}. Due to ergodicity, we observe the convergence of any density to  the canonical density~$\fcan$. Moreover, for sufficiently large damping, the relaxation of the momentum coordinates is significantly faster than of the position coordinates, which can be seen by considering the associated Fokker--Planck equation:
$$
\frac{\partial f}{\partial t} = (L_{\rm Ham} + \gamma L_{\rm OU})f\,,\quad\text{with}\quad L_{\rm OU}g = \frac{1}{\beta}\Delta_p g + \nabla_p\cdot \left(g M^{-1} p\right).
$$
Thus, higher friction $\gamma$ implies that the Ornstein--Uhlenbeck-part dominates the time evolution. The solution of the Ornstein--Uhlenbeck Fokker--Planck equation is
\begin{equation}\label{OURelaxation}
g(t,p) = \int_{\Pspace} K(t,p,r) g(0,r)dr,
\end{equation}
with
\[
K(t,p,r) = \left(\det(2\pi\beta^{-1}C(t))\right)^{-1/2}\,\exp\left(-\frac{\beta}{2} \left(p-re^{-\gamma t}\right)^{T}C(t)^{-1}\left(p-re^{-\gamma t}\right)\right).
\]
and the covariance matrix
\[
C(t) = M \left({\rm id} - e^{-2\gamma M^{-1}t}\right)\,.
\]
Observe that the time variable~$t$ appears in $K$ always multiplied by~$\gamma$. Thus, the larger the damping~$\gamma$, the more rapidly~$g(t,\cdot)$ tends towards the stationary solution: 
\[
\lim_{t\to\infty} K(t,p,r) = f_\Pspace(p)\,.
\]

This suggests that we can find an \emph{optimal lag time} $\tau$, such that for all $t\geq \tau$ and for all $f:\mathcal{X}\rightarrow \mathbb{R}$
$$
P^t_{\rm Lan}f(q,p)\approx f^t(q)\fcan(q,p)
$$
for some $f^t:\Qspace\rightarrow \mathbb{R}$.

We use this to argue in favor of ``almost-Markovianity'' of~$S^t$: In the following let $t\geq \tau$. For $u:\Qspace\rightarrow \mathbb{R}$ there is an $u^t:\Qspace\rightarrow \mathbb{R}$ such that
$$
P^t_{\rm Lan}\big(u(q)\fcan(q,p)\big)\approx u^t(q)\fcan(q,p).
$$

Using this and the semi-group property of $P^t_{\rm Lan}$, we get
\begin{align*}
S^{2t}u(q) &= \frac{1}{f_{\Qspace}(q)}\int_\Pspace P^{2t}_{\rm Lan}\big(u(q)\fcan(q,p)\big)dp\\
 &\approx \frac{1}{f_{\Qspace}(q)}\int_\Pspace P^t_{\rm Lan}\big(u^t(q)\fcan(q,p)\big)dp\\
 &=S^tu^t(q)\\
 &= S^t\Big( \frac{1}{f_{\Qspace}(q)}\int_\Pspace u^t(q)\fcan(q,p) dp\Big)\\
 &\approx S^t\Big( \frac{1}{f_{\Qspace}(q)}\int_\Pspace P^t_{\rm Lan}\big(u(q)\fcan(q,p)\big) dp\Big)\\
 &=(S^t)^2u(q).
\end{align*}

Inductively, it follows that $S^{nt}\approx (S^t)^n$ for $t\geq\tau$, so in this sense, $S^t$ is almost a semigroup for big enough $t$. As the relaxation rate in (\ref{OURelaxation}) scales with $1/\gamma$, we expect the optimal lag time to do the same.



\paragraph{Extrapolating the restored operator.}
Now assume that $\tau>0$ is small enough to allow $R^\tau$ (or $E^\tau$) to be a reasonable approximation to $S^\tau$. Then
$$
S^{(n\tau)}\approx (S^{\tau})^n\approx (R^\tau)^n.
$$ 

We validate this with a simple numerical example. Using the one-dimensional periodic double-well potential introduced in Section \ref{ssec:perturbation}, we want to compute the second largest eigenvalue $\lambda^1(S^t) < \lambda^0(S^t)=1$, which provides insights into the stability of the two metastable sets, as of Theorem~\ref{metastability}. Note that $\lambda^1(R^t)$ does \emph{not} provide a good approximation for $t$ significantly larger than $\tau$, as the error assymptotics of Corollary \ref{taylorspectralerror} only hold for $t\rightarrow 0$.

With damping $\gamma=5$, a choice of $\tau=1/\gamma=0.2$ seems reasonable, as by visual inspection, $\lambda^1(S^t)$ in this region begins to show exponential decay. Moreover, $R^\tau$ and $E^\tau$ still approximate $S^\tau$ well enough: 
\[\lvert\lambda^1(S^\tau)-\lambda^1(R^\tau)\rvert\approx 0.15\,,\quad \lvert\lambda^1(S^\tau)-\lambda^1(E^\tau)\rvert\approx 0.12\,.
\] 
Figure \ref{fig:extrapolatedeigenvalues} compares $\lambda^1(S^t)$ with $\lambda^1\big(S^\tau\big)^n$, $\lambda^1\big(R^\tau\big)^n$ and $\lambda^1\big(E^\tau\big)^n$ for $n=1,\ldots,10$.
\begin{figure}[h!]\label{fig:extrapolatedeigenvalues}
\centering
\includegraphics{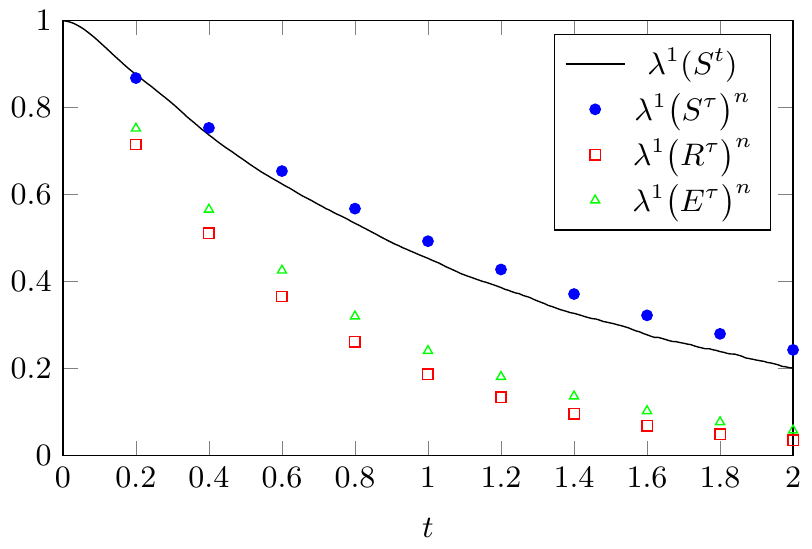}
\caption{Dominant eigenvalue of $S^t$ and its approximations}
\end{figure}
As an error estimate for the eigenvalues, we get
\begin{align*}
\big\lvert\lambda\big(S^{n\tau}\big)-\lambda\big((R^\tau)^n\big)\big\rvert &\leq \big\lvert \lambda\big(S^{n\tau}\big)-\lambda\big((S^\tau)^n\big)\big\rvert + \big\lvert \lambda\big((S^\tau)^n\big)-\lambda\big((R^\tau)^n\big)\big\rvert \\
&\leq \big\lvert \lambda\big(S^{n\tau}\big)-\lambda\big(S^\tau\big)^n \big\rvert + \big\lvert\lambda\big(S^\tau\big)-\lambda\big(R^\tau\big)\big\rvert^n,
\end{align*}
with using the binomial formula to obtain the second inequality. The first term on the right hand side depends on the relaxation of the underlying process after lagtime $\tau$, and (for fixed $n$) decreases with increasing $\tau$. The second term depends on the approximation error of $R^\tau$ on $S^\tau$ and increases with increasing $\tau$. A balance between these two error sources must thus be found. Typically, the optimal lag time lies in the approximation region of $R^t$ and $E^t$ only for high damping $\gamma$. This may or may not correspond to the physical model at hand, and is a significant limitation of the eigenvalue extrapolation method. Alternative restored operators (as proposed in Section~\ref{ssec:MoriZwanzig}) may allow the application for smaller values of~$\gamma$.



\section{Discussion} \label{discussion}

We have considered the dynamics of the position coordinate for a molecular dynamics system given by the Langevin process in thermal equilibrium. After deriving the high friction limit in generalized coordinates, and obtaining the associated Kramers--Smoluchowski dynamics, we have seen that the Smoluchowski equations show up in the evolution of the position coordinates also for any~$\gamma$, in the short time asymptotics after rescaling time according to $t\mapsto  t^2/2$. This can be extended from position coordinates to essential and reaction coordinates (shown for the scalar case here). 
Finally, we discussed possible approaches to stretch the approximation time scales of these pseudo generator methods. Here we argued that these time scales on which a good approximation has to be achieved, are actually finite due to the ergodicity of the Langevin process, and their upper bound decreases with increasing damping~$\gamma$.

%
%
%
%


The numerical experiments in~\cite{BiKoJu14} suggest that our theoretical findings on the asymptotic approximation error can be extended to the dominant spectrum as well, hence that the approach is applicable for metastability analysis.  In order to be applicable to bio-chemically relevant systems, two main points have to be addressed: 
\begin{itemize}
\item[(a)] \emph{Extension of the approximation quality for larger time scales.}

An important aspect of the pseudo generator approach is that it gives a practical tool to systematically derive coarse-grained models of molecular motion by projecting the dynamics onto a subspace of essential coordinates.     
It is yet unclear, however, how the projection onto essential
coordinates influences the approximation quality of the projected pseudo generator $G_{2}^{\rm ess}$ on the original process. We expect the dominant eigenfunctions of
$G_{2}^{\rm ess}$ to be usable to reliably identify metastable sets,
if the selected essential coordinates are slow-moving in comparison to
the "non-essential" coordinates. 
A perturbation expansion in the style
of Section~\ref{ssec:perturbation} might be be able to provide a
rigorous error estimate, and identify the role of the non-essential
coordinates in the approximation. Also, the involvement of higher order derivatives of~$S_{\rm ess}^t$ in the approximation scheme seems promising  (cf.~\cite{ChHaKu00}).
Taking into account higher order terms in the derivation of the pseudo generators seems especially useful when accurate coarse-grained diffusion models in terms of few collective variables are sought in cases when no explicit small parameter is available and therefore traditional averaging or homogenization methods to eliminate unresolved degrees cannot be applied \cite{LeLe10,LuVE14}.\\      

\item[(b)] \emph{Numerical discretization.}

We derived a differential operator expression for projected
pseudo-generators in essential coordinates  ($G_2^{\rm ess}$, cf.\ Lemma~\ref{lemma:projectedpseudogens}) and saw that this operator has a
simple, closed form that can again be interpreted as the generator of a diffusion process. {  Its discretization, especially for multidimensional reaction coordinates, can be conveniently done via spectral collocation using the Feynman--Kac representation of the underlying partial differential equation or the transfer operator, depending on which type of problem is considered; for details, see e.g.~\cite{Mai09,Fro13,ScSa15}. We should stress, however, that while for the unprojected operator,~$G_2$, the collocation matrix can be set up analytically~\cite{BiKoJu14},
for the projected one, $G_2^{\rm ess}$, high-dimensional integrals over non-essential
degrees of freedom are involved. Sampling-based quadrature seems to be
the natural treatment here (see  \cite{CiLeVE08,Har08,FEBook}).}

Further, even if the reduction to a comparatively small number of reaction coordinates can be carried out, the discretization of the corresponding pseudo generators will become computationally challenging due to the curse of dimension if there are more than, say, six of these degrees of freedom.  On the other hand, the macroscopic dynamics of the molecular system is taking place on an essentially one-dimensional skeleton: Apart from motion within metastable states (i.e.\ the conformations of the molecule), the vast majority of conformational transitions occurs along a few dominant, low dimensional transition pathways \cite{EVE04,EVE06,EVE10}.  While metastable states correspond to densities which are almost fixed points under the action of some transfer operator, the transitions can be modelled as curves in the space of densities. This picture alludes to numerical techniques for computing low-dimensional (invariant) sets in systems with higher dimensional state spaces \cite{DeHo97,Del99}, using ansatz functions of higher smoothness in combination with a meshfree approach \cite{Web12}.
\end{itemize}



\appendix

\section{Coordinate expressions for the Smoluchowski equation}

In order to compute the right hand side in (\ref{bweEff}) explicitly, we 
observe that, for functions $\psi=\psi(\bm{u},t)$, 
\begin{eqnarray*}
A_{\text{Ham}} \psi = (G^{-1}\bm{v})\cdot \nabla \psi
\end{eqnarray*}
where $\nabla \psi(\bm{u},t)$ is understood as the derivative 
\wrt the spatial argument, here $\bm{u}$. Upon noting that
\begin{eqnarray*}
A_{\text{OU}}\left(\gamma^{-1}\bm{v} \right) = -G^{-1}\bm{v}\,,
\end{eqnarray*}
with $A_{\text{OU}}$ acting component-wise from the left, we find that 
\begin{eqnarray*}
A_{\text{OU}}^{-1}A_{\text{Ham}} \psi = -(\gamma^{-1}\bm{v})\cdot \nabla \psi
\end{eqnarray*}
for the action of $A_{\text{OU}}^{-1}$ on $A_{\text{Ham}} \psi\in{\rm ran}\,A_{\text{OU}}$. Therefore
\begin{eqnarray*}
A_{\text{Ham}}A_{\text{OU}}^{-1}A_{\text{Ham}}\psi  &=& \sum_{i,j}\gamma^{ij}\left[\frac{\partial}{\partial u_j}
\left(V + \frac{1}{2}\bm{v}\cdot G^{-1}\bm{v} \right)\right]\frac{\partial \psi}{\partial u_i}\\
& & - \sum_{i,j}\left(G^{-1}\bm{v}\right)_i\left[\frac{\partial}{\partial u_i}
\left(\gamma^{-1}\bm{v}\right)_j\right]\frac{\partial \psi}{\partial u_j}\\
 && -\sum_{i,j} \left(G^{-1}\bm{v}\right)_i \left(\gamma^{-1}\bm{v}\right)_j \frac{\partial^2\psi}{\partial u_j u_i}\,,
\end{eqnarray*}
where upper indices indicate inverse matrices, i.e., 
$\gamma^{ij}=(\gamma^{-1})_{ij}$. Using the identity
\begin{eqnarray*}
\int_{\R^{d}}\bm{v}\cdot B\bm{v}\,\varrho_{\bm{u}}(\bm{v})\,d\bm{v} = \frac{1}{\beta}\tr(GB)\,,\quad B\in\R^{d\times d}\,,
\end{eqnarray*}
with $\varrho_{\bm{u}}$ as given by (\ref{rhou}), 
we can easily evaluate the integral in (\ref{bweEff}), which yields 
\begin{eqnarray*}
\bar{A}\psi =   \sum_{i,j}\left[\gamma^{ij}
\left(-\frac{\partial V}{\partial u_j}
+ \frac{1}{2\beta}\tr\left(G^{-1}
\frac{\partial G}{\partial u_j}\right)\right)
\frac{\partial\psi}{\partial u_i}+
\frac{1}{\beta}
\left(\frac{\partial \gamma^{ij}}{\partial u_j}
\frac{\partial}{\partial u_i} +
\gamma^{ij}\frac{\partial^2 \psi}{\partial u_i u_j}\right)\right]\,. 
\end{eqnarray*}
In the last equation we have used the shorthand
\[
\bar{A}\psi = -\int_{\Pspace} \left(A_{\text{Ham}}A_{\text{OU}}^{-1}A_{\text{Ham}}\psi\right)\varrho_{\bm{u}}(\bm{v})\,d\bm{v}\,,
\]
Employing Jacobi's formula $(\det G)'=\det G\,{\rm{tr}}(G^{-1}G')$ and the fact that $\det G=\det M \det(\nabla\Phi^{T}\nabla\Phi)$, the 
above expression for $\bar{A}$ can be recast as desired: 
\begin{equation*}
\bar{A} = \beta^{-1}\tilde{\Delta} - {\nabla} V\cdot \tilde{\nabla}\,,
\end{equation*}
where
\begin{equation*}
\tilde{\nabla}  = \gamma^{-1}  \nabla \quad\mbox{ and }\quad
\tilde{\Delta}=\frac{1}{\sqrt{\det \gamma}} \nabla\cdot
\left(\sqrt{\det \gamma}\, \gamma^{-1} \nabla\,\right)\,,
\end{equation*}
denote gradient and Laplace-Beltrami operator \wrt $\gamma$. 
Note that $\bar{A}$ no longer depends on the constant mass matrix $M$.

\bibliographystyle{abbrv}
\bibliography{bittracher_et_al}
\end{document}